\documentclass[11pt]{amsart}

\usepackage{comment}
\usepackage{color}
\usepackage{amsmath,mathabx}
\usepackage{amsthm}
\usepackage{amssymb}
\usepackage{graphicx}
\usepackage{enumerate}
\usepackage{amsfonts}
\usepackage{palatino}
\usepackage{mathrsfs}
\usepackage{mathdots}
\usepackage{color}

\numberwithin{equation}{section}
\theoremstyle{plain}

\newtheorem{Lemma}{Lemma}[section]

\newtheorem{Corollary}[Lemma]{Corollary}
\newtheorem*{Corollary*}{Corollary}
\newtheorem{Theorem}[Lemma]{Theorem}
\newtheorem*{Theorem*}{Theorem}
\theoremstyle{definition}

\newtheorem{Example}[Lemma]{Example}
\newtheorem{Remark}[Lemma]{Remark}
\newtheorem{Question}[Lemma]{Question}

\allowdisplaybreaks

\usepackage{enumitem}
\setlist[enumerate,1]{label=(\alph*),font=\upshape}

\setlist[enumerate,2]{label=(\roman*),font=\upshape}

\def\HH{\mathscr{H}}

\def\P{\mathscr{P}}

\def\C{\mathbb{C}}

\def\D{\mathbb{D}}
\def\T{\mathbb{T}}

\def\phi{\varphi}

\newcommand{\supess}{\operatorname{supess}}

\renewcommand{\dim}{\operatorname{dim}}

\newcommand{\beqa}{\begin{eqnarray*}}
\newcommand{\eeqa}{\end{eqnarray*}}

\renewcommand{\leq}{\leqslant}

\renewcommand{\subset}{\subseteq}

\title[Composition operators on de Branges--Rovnyak spaces ]{Composition operators on de Branges--Rovnyak spaces associated to a rational (not inner) function}

\author[Alhajj]{Rim Alhajj}
\address{Laboratoire Paul Painlev\'e, Universit\'e de Lille, 59 655 Villeneuve d'Ascq C\'edex }
\email{Rim.hajj@hotmail.com}

\author[Fricain]{Emmanuel Fricain}
 \address{Laboratoire Paul Painlev\'e, Universit\'e de Lille, 59 655 Villeneuve d'Ascq C\'edex }
 \email{emmanuel.fricain@univ-lille.fr}

\thanks{The authors were supported by Labex CEMPI (ANR-11-LABX-0007-01) and the project FRONT (ANR-17-CE40 - 0021).}


\begin{document}

\begin{abstract}
In this paper, we characterize the boundedness, the compactness and the Hilbert--Schmidt property for composition operators acting from a de Branges--Rovnyak space $\HH(b)$ into itself, when $b$ is a rational function in the closed unit ball of $H^\infty$ (but not a finite Blaschke product). In particular, we extend some of the results obtained by D. Sarason and J.N. Silva in the context of local Dirichlet spaces. 
\end{abstract}

\maketitle

\section{Introduction}
In this paper, we study the composition operators $C_\varphi(f)=f\circ \varphi$ acting on a de Branges--Rovnyak space $\HH(b)$, associated to a function $b$ belonging to the closed unit ball of $H^\infty$ and satisfying $\log(1-|b|)\in L^1(\mathbb T)$. The de Branges--Rovnyak spaces $\HH(b)$ (see the precise definition in Section~\ref{sec2}) have been introduced by L. de Branges and J. Rovnyak in the context of model theory. A whole class of Hilbert space contractions is unitarily equivalent to the restriction of the backward shift operator to $\HH(b)$ for an appropriate $b$ belonging to the closed unit ball of $H^\infty$. \\

The study of composition operators has quite a recent history, but the literature on this subject has grown very quickly and many efforts have been dedicated to characterizing the standard spectral properties in various reproducing kernel Hilbert spaces \cite{MR1397026,MR1237406}. It finds its roots in the pioneering works of E. Nordgren \cite{MR223914} and H.J. Schwartz \cite{MR2618707} in the sixties. The idea is to connect the properties of the operator $C_\varphi$ (boundedness, compactness, Hilbert--Schmidt property,...) with the properties of its symbol $\varphi$. It is well--known that for every analytic self-map $\varphi$ of the open unit disk $\D$, the operator $C_\varphi$ maps boundedly the Hardy space $H^2$ of the unit disk into itself. This is known as the Littlewood subordination principle. The compactness is more subtle and has been characterized by J. Shapiro \cite{MR881273} in terms of the behavior of the Nevanlinna counting function $N_\varphi$ associated to the symbol $\varphi$. Recall that the Nevanlinna counting function $N_\varphi$ is a tool from value distribution theory and is defined by 
\[
N_\varphi(w)=-\sum_{\varphi(z)=w}\log|z|,\quad \mbox{if }w\in\varphi(\D)\setminus\{\varphi(0)\}\quad\mbox{and}\quad N_\varphi(w)=0,\quad \mbox{if }w\in\D\setminus\varphi(\D).
\]

In \cite{MR3411049}, Y. Lyubarskii and E. Malinnikova extended Shapiro's compactness criterion for composition operators $C_\varphi$ acting from $\HH(\Theta)=K_\Theta$ into $H^2$, where $K_\Theta=(\Theta H^2)^\perp$ is the so-called model space associated to an inner function $\Theta$. In \cite{MR3915413}, the second author with M. Karaki and J. Mashreghi generalized some of the results of Y. Lyubarskii and E. Malinnikova for $C_\varphi$ viewed as an operator from $\HH(b)$ into $H^2$, when $b$ is a function belonging to the closed unit ball of $H^\infty$.  Note that, since $\HH(b)$ is contractively contained into $H^2$ (and even is closed in $H^2$ if $b=\Theta$ is inner), then, according to the Littlewood subordination principle, the composition operator $C_\varphi$ always maps boundedly $\HH(b)$ into $H^2$. 

If we require that the operator $C_\varphi$ maps $\HH(b)$ into itself, the situation becomes dramatically more difficult and it imposes severe restrictions on the symbol $\varphi$. It has been studied by J. Mashreghi and M. Shabankhah in \cite{MR3176147,MR3001330} for model spaces $K_\Theta$, mainly when the inner function $\Theta$ is a finite Blaschke product, which implies that the space $K_\Theta$ is of finite dimension. In this case of course, the question of boundedness and compactness reduces to the question of knowing whether $f\circ\varphi\in K_\Theta$, for every $f\in K_\Theta$. \\

 Let us recall that when $b(z)=(1+z)/2$, $z\in\D$, the associated de Branges--Rovnyak space $\HH(b)$ coincides with the local Dirichlet space $\mathcal D(\delta_1)$ (with equivalence of the norms). More generally, if $b \in \operatorname{ball}(H^{\infty})$ is a rational function (and not a finite Blaschke product) such that its pythagorean mate $a$ (see Section 2 for the definition) has only simple zeros on $\T$, then it is known that $\HH(b)=\mathcal D(\mu)$ (with equivalence of the norms), where $\mu$ is a positive discrete measure supported on the set of the zeros of $a$ on $\mathbb T$. See \cite{MR3110499}. Here $\mathcal D(\mu)$ is the space of holomorphic functions on $\mathbb D$ whose derivatives are square-integrable when weighted against the Poisson integral of the measure $\mu$. In \cite{MR1945292}, D. Sarason and J.N. Silva studied the composition operators on $\mathcal D(\mu)$. They gave a criterion for the boundedness, the compactness and the Hilbert--Schmidt property of $C_\varphi$ on $\mathcal D(\delta_1)$ in terms of the behavior of a counting function appropriate for the space $\mathcal D(\delta_1)$. Of course, their results can be translated immediately in the context of $\HH(b)$ space where $b(z)=(1+z)/2$. \\

One of the main difficulties when we deal with $\HH(b)$ spaces is to check if a given function $f$ belongs or not to $\HH(b)$. Contrary to most of the classical spaces (Hardy space, Bergman space, Dirichlet space,...), the membership to $\HH(b)$ cannot be characterized directly by an integral condition (at least it is not known). That causes some difficulties to check if a composition operator $C_\varphi$ maps $\HH(b)$ into itself.  In this paper, we will restrict ourselves to the case when $b$ is a rational function in the closed unit ball of $H^\infty$, which is not a finite Blaschke product. In this case, we have a concrete description of the $\HH(b)$ space which makes the situation more tractable. In particular, the key point in our method is a link we establish between the properties of a composition operator $C_\varphi:\HH(b)\longrightarrow \HH(b)$ and the properties of some related weighted composition operator $W_{u,\varphi}(f)=u (f\circ \varphi)$, acting on $H^2$. Here $u$ will be some appropriate function depending on the values of $\varphi$ at the zeros of the pythagorean mate $a$ of $b$. Then, using known results on $W_{u,\varphi}$, we characterize the boundedness, the compactness and the Hilbert--Schmidt property of $C_\varphi$ on $\HH(b)$.

 Let us mention that all the results we obtain in this paper can be translated in the context of $\mathcal D(\mu)$ spaces where $\mu$ is a finite sum of Dirac measures, and we then extend or recover many results obtained by D. Sarason and J.N. Silva in the context of local Dirichlet space $\mathcal D(\delta_1)$. However, our two approaches are different. Indeed, D. Sarason and J.N. Silva used an approach based on an appropriate counting function, whereas we use an approach based on an interesting link with some weighted composition operators. \\

In Section~\ref{sec2}, we present a quick overview of some known properties of de Branges--Rovnyak spaces, useful for our study of composition operators. We also prove a new result on multipliers which is interesting in its own right.  In Section~\ref{sec3}, we give some necessary/sufficient conditions for boundedness. In particular, we prove that the boundedness of $C_\varphi$ imposes some restrictive conditions on the values of $\varphi$ at the zeros of $a$ on $\T$. Section~\ref{sec4} contains a characterization for the boundedness, whereas, in the last section, we give a characterization for the compactness and the Hilbert--Schmidt property. As we will see, major differences between the $H^2$ case and the $\HH(b)$ case emerge, and in particular, we exhibit some examples of symbols $\varphi$ for which $C_\varphi$ is Hilbert--Schimdt on $\HH(b)$ but not compact on $H^2$.  We also discuss interesting connections with angular derivatives in the sense of Carath\'eodory. 

 \section{Preliminaries on $\HH(b)$ spaces}\label{sec2}
 \subsection{Definition of de Branges--Rovnyak spaces}
Let  
 $$\operatorname{ball}(H^{\infty}) := \Big\{b \in H^{\infty}: \|b\|_{\infty} = \sup_{z \in \D} |b(z)| \leq 1\Big\}$$
 be the closed unit ball of $H^\infty$, the space of bounded analytic functions on the open unit disk $\mathbb D=\{z\in\C:|z|<1\}$, endowed with the sup norm. For $b \in \operatorname{ball}(H^{\infty})$, the \emph{de Branges--Rovnyak space} $\HH(b)$ is the reproducing kernel Hilbert space on $\mathbb D$ associated with the positive definite kernel $k_\lambda^b$, $\lambda\in\D$, defined as
\begin{equation}\label{eq:original kernel H(b)}
k^{b}_{\lambda}(z) = \frac{1 - \overline{b(\lambda)}b(z)}{1 - \overline{\lambda} z}, \quad  z \in \D.
\end{equation}
 We refer the reader to the book \cite{Sa} by D. Sarason and to the recent monograph \cite{MR3617311,MR3497010} by the second author and J. Mashreghi for an in-depth study of de Branges--Rovnyak spaces and their connections to numerous other topics in operator theory and complex analysis.
 
It is known that $\HH(b)$ is contractively contained in the Hardy space $H^2$ of analytic functions $f$ on $ \D $ for which 
$$\|f\|_{H^2} := \Big(\sup_{0 < r < 1} \int_{\T} |f(r \xi)|^2 dm(\xi)\Big)^{\frac{1}{2}}<\infty,$$
where $m$ is the normalized Lebesgue measure on the unit circle $\T = \{\xi \in \C: |\xi| = 1\}$. For $f \in H^2$, the non-tangential limit 
$f(\xi):=\lim_{\substack{z\to\xi\\  \sphericalangle}} f(z)$ exists for $m$-almost every $\xi \in \T$ and
\[
\|f\|_{H^2}  = \Big(\int_{\T} |f(\xi)|^2 dm(\xi)\Big)^{\frac{1}{2}}.
\]
See \cite{MR0268655, garnett}. 
Though $\HH(b)$ is contractively contained in $H^2$, it is generally not closed in the $H^2$ norm. Indeed, $\HH(b)$ is closed in $H^2$ if and only if $b=I$ is an inner function, meaning that $|I(\xi)|=1$ for a.e. $\xi\in\mathbb T$. In this case, $\HH(b)=K_I=(I H^2)^\perp$ is the so-called \emph{model space} associated to $I$. In the particular case when $I=B$ is a finite Blaschke product associated to the sequence $\lambda_1,\dots,\lambda_n$ (with multiplicities $m_1,\dots,m_n$), meaning that 
\[
B(z)=\gamma \prod_{j=1}^n \left(\frac{z-\lambda_j}{1-\overline{\lambda_j}z}\right)^{m_j},\qquad z\in\D,
\]
where, for $1\leq j\leq n$, $\lambda_j\in\D$, $m_j\in\mathbb N$  and $\gamma\in\T$, then we have an explicit description of $K_B$ given by 
\begin{equation}\label{descrip-model-space-FBP}
K_B=\left\{\frac{p}{\prod_{j=1}^n (1-\overline{\lambda_j}z)^{m_j}}:p\in \P_{N - 1}\right\},
\end{equation}
where $N=\sum_{j=1}^n m_j$ and $\P_{N-1}$ denotes the set of polynomials of degree less or equal to $N-1$. See \cite[Corollary 5.18]{MR3526203}.\\
%In that paper, we will assume that $b$ is a non-extreme point of $\operatorname{ball}(H^{\infty})$, which is equivalent to $\log(1-|b|)\in L^1(\mathbb T)$. Under that assumption, there is a unique outer function $a$, called the \emph{pythagorean mate} for $b$, such that $a(0)>0$ and $|a|^2+|b|^2=1$ a.e. on $\mathbb T$. 

Another particular case is when $\|b\|_\infty<1$. In this case, the space $\HH(b)$ coincides with the Hardy space $H^2$ as sets (with an equivalent norm). As already mentioned in the introduction, the problem of boundedness and compactness for $C_\varphi$ on $H^2$ has been completely solved. Therefore, {\bf we will assume in the rest of the paper that $\|b\|_\infty=1$}.

\subsection{A description of $\HH(b)$ when $b$ is a rational function}
 
Although the contents of the space $\HH(b)$ may seem mysterious for a general $b \in \operatorname{ball}(H^{\infty})$, it turns out that when $b \in \operatorname{ball}(H^{\infty})$ is a rational function (and not a finite Blaschke product)  the description of $\HH(b)$ is quite explicit. Such a $b$ is a non-extreme point of $\operatorname{ball}(H^{\infty})$, which is equivalent to $\log(1-|b|)\in L^1(\mathbb T)$, and so there is a unique outer function $a$, called the \emph{pythagorean mate} for $b$, such that $a(0)>0$ and $|a|^2+|b|^2=1$ a.e. on $\mathbb T$. When $b$ is rational, $a$ is also a rational function and can be obtained from the Fej\'{e}r--Riesz theorem. See \cite{MR3503356}.

Let $ \xi_1, \dots, \xi_n $ denote the {\em distinct} roots of $ a $ on $ \T $, with corresponding  multiplicities $ m_1, \dots, m_n $, and define the polynomial $ a_1 $ by 
\begin{equation}\label{eq:definition of a}
a_1(z): =\prod_{j=1}^n (z-\xi_j)^{m_j}.
\end{equation}
% When $a$ is the Pythagorean mate for $b$, let 
%\begin{equation}\label{nmnbhjhjk898}
%\widehat{a} = \prod_{j = 1}^{n} (z - \xi_j)^{m_j}.
%\end{equation}
% Note that $\MM(\overline{a}) = \MM(\overline{\widehat{a}})$ \cite[Prop.~2.7]{MR3967886}. 
%We start with a fixed rational, nonextreme function $ b $ in the unit ball of $ H^\infty $. The Pytagorean mate of $ b $ is the outer rational function $ a' $ with the property that $ |a'(\xi)|^2+|b(\xi)|^2=1 $ almost everywhere on $ \T $. We use this slightly unusual notation because we are interested not in $ a' $ but in the function 
Observe that since we assume that $\|b\|_\infty=1$, the function $a$ has necessarily some zeros on $\T$ (corresponding to the points where $b$ achieves its maximum on the closed unit disk). Results from  \cite{MR3110499, MR3503356} show that $\HH(b)$ has an explicit description as 
\begin{equation}\label{eq:formula for H(b)}
\HH(b)=a_1H^2 \oplus \P_{N-1},
\end{equation}
where $ N=m_1+\dots+m_n $ and $\oplus$ above denotes a topological direct sum in $\HH(b)$. Moreover, if $ f\in\HH(b) $ is decomposed with respect to \eqref{eq:formula for H(b)} as 
\begin{equation}\label{uUUiipPPS}
 f=a_1\widetilde{f}+p_f, \quad  \mbox{where $\widetilde{f}  \in H^2$ and  $p_f \in \P_{N - 1}$},
 \end{equation}
an equivalent norm on $ \HH(b) $  (to the natural one $\|\cdot\|_b$ induced by the positive definite kernel $k_{\lambda}^{b}$, $\lambda\in\D$, above)  is
\begin{equation}\label{eq:norm in h(b)}
\vvvert a_1\widetilde f+p_f\vvvert^{2}_{b}:=\|\widetilde{f}\|^2_{H^2}+\|p_f\|^2_{H^2}.
\end{equation}
Also we use the scalar product $\langle\cdot,\cdot\rangle_b$ associated to $\vvvert\cdot\vvvert_b$ defined as
\begin{equation}\label{eq:scalar-product in h(b)}
\langle f_1 f_2\rangle_b=\langle \widetilde f_1,\widetilde f_2\rangle_2+\langle p_{f_1},p_{f_2}\rangle_2,
\end{equation}
for every $f_1=a_1\widetilde f_1+p_{f_1}$, $f_2=a_1\widetilde f_2+p_{f_2}$ in $\HH(b)$. In particular, with the scalar product given by \eqref{eq:scalar-product in h(b)}, the direct sum in \eqref{eq:formula for H(b)} becomes an orthogonal sum. It is important to note that $ \vvvert \cdot\vvvert_b $ is only equivalent to the original norm $\|\cdot\|_b$ associated to the kernel in~\eqref{eq:original kernel H(b)}, and its scalar product as well as the reproducing kernels and the adjoints of operators defined on $\HH(b)$ will be different. However, the boundedness and compactness properties for the operator $C_\varphi$ on $\HH(b)$ do not depend on the equivalent norm we consider. So in the rational case, there is no problem to work  with the norm given by \eqref{eq:norm in h(b)}  and the scalar product given by \eqref{eq:scalar-product in h(b)}, which {\bf{we will do in the rest of the paper}}. 

Notice that we exclude the case when $b$ is a finite Blaschke product because, as already mentioned, this case has already been studied in \cite{MR3176147,MR3001330} for the boundedness, and since the associated $\HH(b)$ space is of finite dimension (when $b$ is a finite Blaschke product), the problem of the compactness for the operator $C_\varphi$ on $\HH(b)$ reduces to the problem of its boundedness. 
 
% %Since \eqref{eq:norm in h(b)} defines the norm and the scalar product on $ \HH(b) $ that we will use in the sequel, we prefer
% With the norm $\|\cdot\|_{b}$ and the corresponding inner product in mind, we need
%  to introduce a new notation for the associated reproducing kernels, different from~\eqref{eq:original kernel H(b)}, namely $ \kk^b_\lambda $ (note the bold face). By the term reproducing kernel we mean that $\kk^{b}_{\lambda} \in \HH(b)$ for all $\lambda \in \D$ and 
% \[
% \langle f, \kk^b_\lambda\rangle_b = f(\lambda) \quad \mbox{for all $ f\in\HH(b) $ and $ \lambda\in\D $.}
% \]
 
An important property of functions $f$ in $\HH(b)$, when $b\in\operatorname{ball}(H^{\infty})$ is rational (but not a finite Blaschke product), is the existence of non-tangential limits for $f$ and some of its derivatives at certain points on $\T$. More precisely, for every  $1 \leq j \leq n$, for every $0\leq k\leq m_j-1$ and every $f \in \HH(b)$, we have 
 \begin{equation}\label{q}
 f^{(k)}(\xi_j):= \lim_{\substack{z\to\xi_j\\  \sphericalangle}} f^{(k)}(z) = p_f^{(k)}(\xi_j),
 \end{equation}
 where $f = a_1 \widetilde{f} + p_f$ with $\widetilde{f} \in H^2$ and $p_f \in \P_{N - 1}$. See \cite[Corollary 27.22]{MR3617311}.  
 \subsection{The multipliers of $\HH(b)$ spaces when $b$ is rational} 
 
 A tool which will turn out to be useful when studying the boundedness and compactness of composition operators on $\HH(b)$ is the notion of multipliers. Recall that the set $\mathfrak M(\HH(b))$ of multipliers of $\HH(b)$ is defined as 
 \[
 \mathfrak M(\HH(b))=\{\varphi\in\mbox{Hol}(\D):\varphi f\in\HH(b), \forall f\in\HH(b)\}.
 \]
 Using standard arguments, we see that $\mathfrak M(\HH(b))\subset H^\infty\cap\HH(b)$. In general, this inclusion is strict. However, when $b\in\operatorname{ball}(H^{\infty})$ is rational (but not a finite Blaschke product), it is proved in \cite{MR3967886} that we have equality, meaning that 
 \begin{equation}\label{eq:multiplier}
 \mathfrak M(\HH(b))=H^\infty\cap\HH(b).
 \end{equation} 
 
The following result will be useful in our study of boundedness of composition operators and is interesting in its own right.
\begin{Lemma}\label{Lemme:multiplicateur}
 Let $b \in \operatorname{ball}(H^{\infty})$ be a rational function (but not a finite Blaschke product). Let $\varphi\in\HH(b)$ and assume that $1/\varphi\in H^\infty$. Then $1/\varphi\in \mathfrak M(\HH(b))$.
\end{Lemma}

\begin{proof}
Let $a$ be the pythagorean mate of $b$, to which we associate the polynomial $a_1(z)=\prod_{j=1}^n (z-\xi_j)^{m_j}$ as in \eqref{eq:definition of a}. Using \eqref{eq:formula for H(b)}, we can decompose $\varphi\in\HH(b)$ as 
\begin{equation}\label{eq 1.5.5}
 \varphi = a_1 \tilde{ \varphi} + p_{ \varphi},
\end{equation}
 where $ \tilde{ \varphi} \in H^2$ and $p_{ \varphi} \in \P_{N - 1}$. We also know from \eqref{q} that for every $1\leq j\leq n$ and every $0\leq k\leq m_j-1$, the function $\varphi^{(k)}$ has a non-tangential limit at point $\xi_j$ and 
\[ 
\varphi^{(k)} (\xi_j) = p_{ \varphi}^{(k)} (\xi_j).
\]
The assumption $1/\varphi\in H^\infty$ implies the existence of a constant $\delta>0$ such that for every $z\in\D$, $|\varphi(z)|\geq\delta$. In particular, letting $z$ tend non-tangentially to $\xi_j$ gives $\varphi(\xi_j)\neq 0$ for every $1\leq j\leq n$.  

Now set  $h=1/\varphi$. According to \eqref{eq:multiplier} and our assumption, we need to prove that $h\in \HH(b)$. Notice that for every $0 \leq k \leq m_j -1$, we have  
\[
h^{(k)}  = \frac{\psi_k ( \varphi , \varphi^{ \prime}, \ldots, \varphi^{(k)})}{\varphi^{k +1}},
\]
where $ \psi_k$ is a polynomial of $ k+1$ variables. In particular, we deduce that for every $1 \leq j \leq n$ and every $0 \leq k \leq m_j - 1$, the function $h^{(k)}$ has a non-tangential limit at $\xi_j$. Consider now the (unique) polynomial $p_h \in \P_{N-1}$ such that 
\[
p_h^{(k)} (\xi_j) = h^{(k)} (\xi_j),\qquad 1\leq j\leq n,\,0\leq k\leq m_j-1.
\]
The polynomial $p_h$ can be constructed using Hermite interpolating polynomials. See for instance \cite[Chap. 1, E.7]{MR1367960}. According to \eqref{eq:formula for H(b)}, we need to check that the function $ \psi := \dfrac{h - p_h}{a_1}$ is in $H^2$ in order to conclude that $h\in\HH(b)$. To this purpose, write 
\begin{align*}
\psi &= \dfrac{\frac{1}{ \varphi} - p_h}{a_1} = \dfrac{1}{ \varphi} . \dfrac{1 - \varphi p_h}{a_1}\\
&= \dfrac{1 }{ \varphi} . \dfrac{1 - p_\varphi p_h}{a_1} + \dfrac{1}{ \varphi} . \dfrac{p_h (p_\varphi - \varphi)}{a_1 }.
\end{align*}
It follows from \eqref{eq 1.5.5} that $(p_{ \varphi} - \varphi)/a_1= - \tilde{\varphi} \in H^2$ and since $1/\varphi$ and $p_h $ are in $H^{ \infty}$, the second term $p_h(p_\varphi-\varphi)/(\varphi a_1)$ is in $H^2$. For the first term, using one more time that  $1/\varphi \in H^{ \infty}$, it is sufficient to prove that $ \dfrac{  p_{ \varphi} p_{h} - 1}{a_1} \in H^2$. To this purpose, observe that for $1\leq j\leq n$, we have
\[
(p_{ \varphi} p_h - 1)(\xi_j) = \varphi(\xi_j) h(\xi_j) - 1 = 0.
\]
Thus every point $\xi_j$, $1\leq j\leq n$, is a zero of the poynomial $p_{ \varphi} p_h - 1$. Moreover, for $1\leq j\leq n$ and  $1 \leq k \leq m_j - 1$,  we have
\begin{align*}
( p_{ \varphi} p_h  - 1)^{(k)} (\xi_j) &= (p_{ \varphi} p_h)^{(k)} (\xi_j)\\
&= \sum_{\ell = 0}^k  \binom{k}{\ell} p_{ \varphi}^{ ( \ell)} (\xi_j) p_h^{(k - \ell)} (\xi_j) \\
&= \sum_{ \ell = 0}^k  \binom{k}{\ell} \varphi^{( \ell)}(\xi_j) h^{ (k - \ell)} (\xi_j) \\
&= ( \varphi h )^{(k)} (\xi_j)\\
&= ( \varphi h - 1  )^{(k)} (\xi_j) = 0,
\end{align*}
because $ \varphi h - 1\equiv 0 $. Hence, for every $1 \leq j\leq n$,  $\xi_j$ is a zero of the polyomial $p_{ \varphi} p_h - 1$ with a multiplicity at least $m_j - 1$. In particular, the polynomial $a_1(z) = \prod_{i=1}^n (z - \zeta_i)^{m_i}$ divides the polynomial $p_{ \varphi} p_h - 1$, meaning that the function $(p_\varphi p_h-1)/a_1$ is also a polynomial and thus belongs to $H^2$. Finally $ \psi \in H^2$ and then $1/\varphi \in \HH(b)$. 
\end{proof}

\begin{Corollary}\label{cor 0.4.0.19}
 Let $b \in \operatorname{ball}(H^{\infty})$ be a rational function (but not a finite Blaschke product). Let $\varphi\in\HH(b)\cap H^{\infty}$. Then, for every  $\lambda \in \mathbb{C}$, $|\lambda|<\|\varphi\|_\infty^{-1}$, the function  $1/(1-\overline{\lambda}\varphi)\in \mathfrak M(\HH(b))$. 
\end{Corollary}
\begin{proof}
This result follows immediately from Lemma~\ref{Lemme:multiplicateur}, because first $1 - \bar{ \lambda} \varphi \in \HH(b)$ and second 
\[
 | 1 - \bar{ \lambda} \varphi (z) | \geq 1 - | \lambda | | \varphi (z) | \geq 1 - | \lambda | \|\varphi\|_\infty> 0,
 \]
which implies that $1/(1-\overline{\lambda}\varphi)\in H^{ \infty}$. 
\end{proof}

\subsection{Two technical results in $\HH(b)$ when $b$ is a rational function}
The next result introduces a function $\psi$, related to the symbol $\varphi$, which will be useful in our characterization for the boundedness of the operator $C_\varphi$ on $\HH(b)$. In order to prove the result, we need to recall a formula on boundary Taylor expansion. Let $h$ be an analytic function on the open unit disk $\D$. Assume that $h,h',\dots h^{(\ell)}$ have non-tangential limits at a point $\xi\in\T$. Then we can write
\begin{equation}\label{eq:taylor}
h(z)=\sum_{k=0}^\ell \frac{h^{(k)}(\xi)}{k!}(z-\xi)^k +(z-\xi)^\ell \varepsilon(z),\qquad z\in\D,
\end{equation}
where $\varepsilon$ is an analytic function on $\D$ with a zero non-tangential limit at $\xi$. A version of this formula appears in \cite[Lemma 22.5]{MR3617311} in the context of the upper half-plane but the proof can be easily adapted to the context of the open unit disk. 

\begin{Lemma}\label{lemme 3.9} 
 Let $b \in \operatorname{ball}(H^{\infty})$ be a rational function (but not a finite Blaschke product), let $a$ be its pythagorean mate to which we associate the polynomial $a_1(z)=\prod_{j=1}^n (z-\xi_j)^{m_j}$ as in \eqref{eq:definition of a}. To each $\varphi\in\HH(b)\cap H^{\infty}$, we associate the function $\psi$ defined as
 \[
 \psi(z)=\prod_{j=1}^n \left(\frac{\varphi(z)-\varphi(\xi_j)}{z-\xi_j}\right)^{m_j},\qquad z\in\D.
 \]
 Then $\psi$ belongs to $H^2$. 
\end{Lemma}

\begin{proof}
Let $h=\prod_{j=1}^n (\varphi-\varphi(\xi_j))^{m_j}$, so that $\psi=h/a_1$. By our assumption and \eqref{eq:multiplier}, the function $h$ belongs to $\HH(b)$. In particular, there exists $\tilde h \in H^2$ and $p_h \in \P_{N-1}$ such that $h = a_1 \tilde h + p_h$. It follows from \cite[Corollary 27.22]{MR3617311} that
\begin{equation}\label{eq:03238-lemme 3.9}
p_h = \sum_{i=1}^n \sum_{k=0}^{m_i - 1} h^{(k)} ( \xi_i) r_{i, k} ,
\end{equation}
where $(r_{i,k})_{1 \leq i \leq n,\,0 \leq k \leq m_i - 1}$ are the Hermite polynomials of degree less or equal to $N-1$ such that
$$r_{i,k}^{(\ell)} ( \xi_j)  = \left\{
    \begin{array}{ll}
        1 & \mbox{if $i=j$ and $k= \ell$} \\
        0 & \mbox{otherwise } 
    \end{array}.
\right.$$
Observe that it is sufficient to prove that
\begin{equation}\label{eq 3.10} 
h^{(k)} ( \xi_i) = 0, \qquad \mbox{for every }1 \leq i \leq n \hspace{0.1 cm}\mbox{and } 0 \leq k \leq m_i - 1.
\end{equation}
Indeed, it follows from \eqref{eq:03238-lemme 3.9} and \eqref{eq 3.10} that $p_h= 0$, which implies that 
$$ \psi=\frac{h}{a_1}=\tilde h,$$
whence $\psi \in H^2$.\

In order to prove \eqref{eq 3.10}, fix $1\leq i\leq n$. On one hand, observe that 
\begin{equation}\label{eq:forme-h-lemme3.10}
h=(\varphi-\varphi(\xi_i))^{m_i}\psi_i,
\end{equation} 
where $\psi_i$ is an analytic function which has a non-tangential limit at $\xi_i$. Moreover, according to \eqref{eq:taylor}, we can write 
\[
\varphi(z)=\varphi(\xi_i)+(z-\xi_i)\varphi'(\xi_i)+(z-\xi_i)\varepsilon(z),
\]
where $\varepsilon$ is an analytic function on $\D$ which has a zero non-tangential limit at point $\xi_i$. Thus, we get
\[
h(z)=\left((z-\xi_i)\varphi'(\xi_i)+(z-\xi_i)\varepsilon(z)\right)^{m_i} \psi_i(z)=(z-\xi_i)^{m_i}\left(\varphi'(\xi_i)+\varepsilon(z)\right)^{m_i}\psi_i(z).
\]
On the other hand, since $h\in\HH(b)$, using \eqref{q} and one more time \eqref{eq:taylor}, we also have
\[
h(z)=\sum_{k=0}^{m_i-1}\frac{h^{(k)}(\xi_i)}{k!}(z-\xi_i)^k+(z-\xi_i)^{m_i-1}\varepsilon_1(z),
\]
where $\varepsilon_1$ is an analytic function on $\D$ which has a zero non-tangential limit at point $\xi_i$. Therefore, we deduce that 
\[
(z-\xi_i)^{m_i}\left(\varphi'(\xi_i)+\varepsilon(z)\right)^{m_i}\psi_i(z)=\sum_{k=0}^{m_i-1}\frac{h^{(k)}(\xi_i)}{k!}(z-\xi_i)^k+(z-\xi_i)^{m_i-1}\varepsilon_1(z).
\]
It is now easy to see that this identity implies that $h^{(k)}(\xi_i)=0$ for $0\leq k\leq m_i-1$, which concludes the proof of \eqref{eq 3.10}. 
\end{proof} 
 
 We end this section with a result on some particular subspaces of $\HH(b)$ which will be of use to us in our study of compactness. It will enable us to restrict our composition operators on some subspaces of finite codimension.  
 
\begin{Lemma}\label{lemme 0.4.0.24}
 Let $b \in \operatorname{ball}(H^{\infty})$ be a rational function (but not a finite Blaschke product), let $a$ be its pythagorean mate to which we associate the polynomial $a_1$ as in \eqref{eq:definition of a}. For every inner function $I$, the subspaces $a_1 I H^2$ and $a_1 K_I$ are closed in $\HH(b)$. Moreover, we have
 \begin{equation}\label{eq 3.18} 
 \HH(b) \ominus_b a_1 I H^2 = a_1 K_{I} \oplus^\perp_b\P_{N-1}.
\end{equation} 
In particular, when $I$ is a finite Blaschke product, the subspace $a_1 I H^2$ has a finite codimension in $\HH(b)$.
\end{Lemma} 

Here the notation $\ominus_b$ denotes the orthogonal complement in $\HH(b)$ with respect to the scalar product defined in \eqref{eq:scalar-product in h(b)}. 
 \begin{proof}
Let us introduce the operator $V : H^2  \rightarrow \HH(b)$ defined by $V(f) = a_1 f$, $f \in H^2$. According to \eqref{eq:norm in h(b)}, the operator $V$ is an isometry from $H^2$ into $\HH(b)$. Since the subspaces $K_I$ and $I H^2$ are closed subspaces of $H^2$, their ranges under $V$, respectively $a_1 K_I$ and $a_1 I H^2$, are closed in $\HH(b)$. 

Let us now check \eqref{eq 3.18}. To this purpose, let $f \in \HH(b)$. According to \eqref{uUUiipPPS}, write $f$ as  $f = a_1 \tilde{f} + p_f$ where $\tilde{f} \in H^2$ and $p_f \in \P_{N-1}$. Using \eqref{eq:scalar-product in h(b)}, the function $f\in \HH(b) \ominus_b a_1 I H^2$, if and only if for every $h\in H^2$, we have 
\[
0 = \langle f , a_1 I h\rangle_b = \langle a_1 \tilde{f} + p_f , a_1 I h\rangle_b=\langle \tilde{f}, I h\rangle_2,
\]
which is equivalent to $\tilde{f}\in H^2\ominus IH^2=K_I$. In other words, $f \in a_1 K_{I} \oplus \P_{N-1}$, which proves \eqref{eq 3.18}.

Now if $I$ is a finite  Blaschke product, then $\dim ( K_{I}) < \infty $ (see \cite[Section 14.2]{MR3497010}). Hence $ \dim (a_1 K_{I} \oplus_b^\perp \P_{N-1}) < \infty$, and the conclusion now follows from \eqref{eq 3.18}.
 \end{proof}

\section{Some basic necessary/sufficient conditions for boundedness}\label{sec3}
The aim of this section and the following is to study the boundedness of composition operators $C_\varphi:\HH(b)\longrightarrow\HH(b)$, where $C_\varphi(f)=f\circ \varphi$, $f\in\HH(b)$. Since $\HH(b)$ is a space of analytic functions on $\D$, it is necessary to require that $\varphi:\D\longrightarrow \D$ is an analytic self-map on $\D$, which is equivalent to require that $\varphi \in\operatorname{ball}(H^{\infty})$. 

\subsection{The general case}
We start with a very standard result in the theory of composition operators on reproducing kernel Hilbert spaces but for completeness, we give a proof in our context. 
\begin{Lemma}\label{lemme 4.2.1.1}
Let $b$ and $\varphi$ belong to $\operatorname{ball}(H^{\infty})$. Assume that $C_{ \varphi} ( \HH(b)) \subset \HH(b)$. Then the followings hold:
\begin{enumerate}
\item $C_{ \varphi}$ is a bounded operator on $ \HH(b)$. 
\item If furthermore $b$ is a non-extreme point of $\operatorname{ball}(H^{\infty})$, then $\varphi$ belongs to $\HH(b)$. 
\end{enumerate}
\end{Lemma}

\begin{proof}
(a) We apply the closed graph theorem. Let $(f_k)_k$ be a sequence in $\HH(b)$, and assume that $f_k \longrightarrow f$ as $k \rightarrow \infty$, in $ \HH(b)$ and $f_k \circ \varphi \longrightarrow g$ as $k \rightarrow \infty$, in $ \HH(b)$. We need to show that $f\circ\varphi=g$. Since the convergence in $\HH(b)$ implies the pointwise convergence, for every $\lambda\in\D$, we have 
 $(f_k \circ \varphi)( \lambda) \longrightarrow g( \lambda )$, as $k \rightarrow \infty$.  But $ (f_k \circ \varphi)( \lambda)=f_k(\varphi ( \lambda))$ and $\varphi ( \lambda) \in \mathbb{D}$. Therefore, we also have $f_k ( \varphi ( \lambda )) \longrightarrow f( \varphi ( \lambda ))$ as  $k \rightarrow \infty$. By unicity of the limit, we then deduce that for every $ \lambda \in \mathbb{D}$,  $f( \varphi ( \lambda ))=g( \lambda )$. In other words, $f \circ \varphi = g$, which by the closed graph theorem implies that  $C_{ \varphi}$ is bounded from $\HH(b)$ into itself. \\
(b) Observe that when $b$ is a non-extreme point of the closed unit ball of $H^{  \infty}$, the polynomials belong to $\HH(b)$ (see \cite[Theorem 23.13]{MR3617311}). In particular, the identity function $e_1$, defined by $e_1 (z) = z$, $z\in\D$, is in $ \HH(b)$, which immediately implies that $\varphi=C_\varphi(e_1)\in\HH(b)$. 
\end{proof}

It is not surprising that the problem for the boundedness of $C_\varphi$ may occur when the symbol $\varphi$ touches the boundary. In the opposite case, if we assume furthermore that $\varphi$ extends analytically trough the closed unit disc, we prove that the operator $C_\varphi$ is always bounded on $\HH(b)$, when $b$ is a non-extreme point of $\operatorname{ball}(H^{\infty})$.

 First, let us recall that when $b$ is a non-extreme point of $\operatorname{ball}(H^{\infty})$, then $\mbox{Hol}(\overline{\D})$, the space of functions which are analytic in a neighborhood of the closed unit disk $\overline{\D}$, is contained in $\HH(b)$ (see \cite[Theorem 24.6]{MR3617311}). 

\begin{Lemma}\label{lem-fonction-analytique-closed-unit-disk}
Let $b$ be a non-extreme point of $\operatorname{ball}(H^{\infty})$ and let $\varphi\in \mbox{Hol}(\overline{\D})$ such that $\varphi(\overline{\D})\subset\D$. Then $C_\varphi$ is bounded on $\HH(b)$. 
\end{Lemma}

\begin{proof}
Note that, for every $f\in\HH(b)$, we have $f\circ \varphi\in \mbox{Hol}(\overline{\D})\subset \HH(b)$. Then, according to Lemma~\ref{lemme 4.2.1.1}, the operator $C_\varphi$ is bounded from $\HH(b)$ into itself. 
\end{proof}

When $b\in\operatorname{ball}(H^{\infty})$ is rational (but not a finite Blaschke product), we will see, in Section~\ref{sec5}, that the assumptions of Lemma~\ref{lem-fonction-analytique-closed-unit-disk} not only imply that the operator $C_\varphi$ is bounded but even Hilbert-Schmidt on $\HH(b)$. We do not know if this is true for a general non-extreme point $b$ of $\operatorname{ball}(H^{\infty})$. In particular, the following question remains open.
\begin{Question}
Let $b$ be a non-extreme point of $\operatorname{ball}(H^{\infty})$ and let $\varphi\in \mbox{Hol}(\overline{\D})$ such that $\varphi(\overline{\D})\subset\D$. Does it follows that $C_\varphi$ is compact on $\HH(b)$? 
\end{Question}

\subsection{The case when $b$ is a rational function in $\operatorname{ball}(H^{\infty})$} Throughout the rest of this paper, we {\bf now assume} that $b\in\operatorname{ball}(H^{\infty})$ is a rational function (but not a finite Blaschke product) and $\|b\|_\infty=1$. Let $a$ be its pythagorean mate to which we associate the polynomial $a_1(z)=\prod_{j=1}^n (z-\xi_j)^{m_j}$ as in \eqref{eq:definition of a}. We also denote by $Z_\T(a)$ the set of the zeros of $a$ on $\T$, that is 
\[
Z_\T(a)=\{\xi_j:1\leq j\leq n\}.
\]
In order to exhibit some crucial necessary conditions on the symbol $\varphi$ for the boundedness of the composition operator $C_\varphi$ on $\HH(b)$, we need the concept of interpolating sequence for $H^\infty$, which we briefly recall now. We say that a sequence $(z_k)_k$ in $\D$ is an interpolating sequence for $H^\infty$ if, for every bounded sequence $(w_k)_k$, we can find a function $f\in H^\infty$ such that $f(z_k)=w_k$, for every $k\geq 1$. The interpolating sequences for $H^\infty$ have been characterized by L. Carleson \cite{MR117349}. We simply mention a sufficient condition: if there exists $0<q<1$ such that $(z_k)_k$ satisfies $|z_{k-1}|\leq |z_k|$ for every $k\geq 2$, and
\[
\limsup_{k\to\infty}\dfrac{1-|z_k|}{1-|z_{k-1}|}\leq q,
\]
then $(z_k)_k$ is an interpolating sequence for $H^\infty$. See \cite[page 159]{MR827223}.
The following technical and simple result will be useful for our study of composition operators on $\HH(b)$. 
 \begin{Lemma}\label{lemme 4.2.1}
 Let $\varphi:\D\to\D$ be a map which has a radial limit at $\xi\in\T$ satisfying $\varphi(\xi)\in\T$. Then there exists a sequence $(r_k)_k\subset (0,1)$ satisfying $r_k\to 1$, as $k\to \infty$, and the sequence $(\varphi(r_k\xi))_k$ is an interpolating sequence for $H^\infty$. 
 \end{Lemma}
 
 \begin{proof}
 Let $(t_\ell)_\ell \subset (0,1)$, $t_\ell \longrightarrow 1$, as  $\ell \rightarrow \infty$. According to the assumption, we have $| \varphi ( t_\ell \xi)| \longrightarrow 1$ as  $\ell\rightarrow \infty$. Hence by induction, we can construct a subsequence $(t_{\ell_k})_{k \geq 1}$ such that for every $k \geq 1$, we have
$$ 1 - | \varphi (t_{\ell_k} \xi)| \leq \frac{1}{2} ( 1 - | \varphi ( t_{\ell_{k-1}}  \xi)|). $$
Define now $r_k=t_{\ell_k}$ and $z_k := \varphi (r_k \xi )$. Since
$$1 - |z_k| \leq \frac{1}{2} (1 - |z_{k-1}| ) ,$$
we deduce that
$$ \limsup_{k \rightarrow \infty} \frac{1 - |z_k|}{1 - |z_{k - 1}|} \leq \frac{1}{2} < 1, $$
whence $(z_k)_k$ is an interpolating sequence for $H^\infty$.
\end{proof}

The next result now imposes some restrictions on the symbol $\varphi$ for the boundedness of the composition operator $C_\varphi$ on $\HH(b)$. According to Lemma~\ref{lemme 4.2.1.1}, remind that, if $C_\varphi$ is bounded on $\HH(b)$, then $\varphi\in \HH(b)$, and therefore, with \eqref{q}, it follows that $\varphi$ admits a non-tangential limit $\varphi(\xi_j)$ at point $\xi_j$ for every $1\leq j\leq n$. 

\begin{Theorem}\label{lemme 4.0.4.0}
Let $ \varphi : \mathbb{D} \rightarrow \mathbb{D}$ be analytic. If $ C_{ \varphi}$ is bounded on $ \HH(b)$, then for every $j \in \{1,\ldots,n\},$ we have
$$
\varphi ( \xi_j) \in \mathbb{D} \cup Z_\T(a).
$$ 
\end{Theorem}
\begin{proof}
The idea of the proof is similar to \cite[Lemma 4]{MR1945292} (in the context of local Dirichlet space). Observe that, for every $j \in \{ 1, \ldots , n \}$, we have $| \varphi ( \xi_j)| \leq 1$. Assume now that for some $j \in \{ 1, \ldots , n\}$, we have $|\varphi ( \xi_j)| = 1$, and let us show that  $ \varphi ( \xi_j) \in Z_\T(a)$. According to Lemma~\ref{lemme 4.2.1}, we can find a sequence $ (r_k)_k\subset (0,1)$, $r_k{\longrightarrow} 1$, as $k\to\infty$, such that if $z_k := \varphi (r_k \xi_j )$, $k \geq 1$, then $(z_k)_k$  is an interpolating sequence for $H^{ \infty}$. In particular, there is a function $f \in H^{ \infty} $ such that 
$$f(z_k)= f( \varphi (r_k \xi_j))   = \left\{
    \begin{array}{ll}
        1 & \mbox{if } k \hspace{0.1 cm}  \text{is even} \\
        0 & \mbox{if } k \hspace{0.1 cm}  \text{is odd}
    \end{array} .
\right.$$
According to \eqref{eq:formula for H(b)}, the function $a_1 f \in \HH(b)$ and thus $C_{ \varphi} (a_1 f ) = (a_1 \circ \varphi).(f \circ \varphi)$ also belongs to $ \HH(b)$. In particular, $ (a_1 \circ \varphi ). ( f\circ \varphi )$ has a non-tangential limit at $ \xi_j$. Thus $ (a_1 \circ \varphi )(r_k \xi_j)\,( f\circ \varphi ) (r_k \xi_j)$ should have a limit when $k \rightarrow \infty$. But observe that $(f \circ \varphi )( r_k \xi_j ) = f(z_k)$ has two different cluster points and 
\[
 (a_1 \circ \varphi)(r_k \xi_j) = \prod_{i=1}^n ( \varphi (r_k \xi_j) - \xi_i )^{m_i} \longrightarrow \prod_{i=1}^n ( \varphi ( \xi_j) - \xi_i)^{m_i} , \hspace{0.2 cm}\mbox{as } k \rightarrow \infty.
 \]
Thus necessarily, we should have  
$$ \prod_{i=1}^n ( \varphi ( \xi_j) - \xi_i)^{m_i} = 0.$$
Hence there exists $1 \leq i \leq n$ such that $ \varphi ( \xi_j) = \xi_i\in Z_\T(a)$.
\end{proof}

We recover the following result observed by Sarason-Silva \cite{MR1945292} in the context of $\mathcal D(\delta_1)$ space. 
\begin{Corollary}\label{cor 4.2.6}
Let $b(z)=(1+z)/2$ and $\varphi: \mathbb{D} \rightarrow \mathbb{D}$ be analytic. Assume that $C_{ \varphi}$ is bounded on $\HH(b)$. Then either $\varphi(1)\in \mathbb{D}$ or $ \varphi(1)=1$. 
\end{Corollary}
\begin{proof}
It is sufficient to note that the pythagorean mate for $b$ is given by  $a(z)=(1-z)/2$ and  apply Theorem~\ref{lemme 4.0.4.0}.
\end{proof}

In the study of boundedness of $C_{ \varphi}$ on $ \HH(b)$, according to  Lemma~\ref{lemme 4.2.1.1} and Theorem~\ref{lemme 4.0.4.0},  we may assume without loss of generality that $ \varphi \in \HH(b)\cap \operatorname{ball}(H^{\infty})$ and for every $j \in \{1, \ldots , n \} ,$ we have
$$\varphi ( \xi_j) \in \mathbb{D} \cup Z_\T(a). $$

Up to rearranging the sequence $ \xi_j$, $1\leq j \leq n$,  { \bf{we will now assume throughout this paper that}}
\begin{equation}\label{eqa}
\varphi( \xi_j) \in  Z_\T(a)\qquad \mbox{for } 1\leq j\leq p,
\end{equation}
and  
 \begin{equation}\label{eqb}
\varphi( \xi_j) \in \D\qquad  \mbox{for } p+1\leq j\leq n,
\end{equation}
where $0 \leq p \leq n$. When $p=0$, the condition \eqref{eqa} is void and it corresponds to the case when all the non-tangential limits $\varphi(\xi_j)$ belong to $\D$.  When $p=n$, the condition \eqref{eqb} is void and it corresponds to the case when all the non-tangential limits $\varphi(\xi_j)$ belong to $\T$. \\

We end this section by an important necessary condition for the boundedness of the operator $C_\varphi$ on $\HH(b)$ which will be of use to us when making a connection with the boundedness of some related weighted composition operators on $H^2$. 
\begin{Lemma}\label{lemme 0.4.0.17}
Let $ \varphi\in\HH(b)\cap \operatorname{ball}(H^{\infty})$ and assume that $\varphi$ satisfies  \eqref{eqa} and \eqref{eqb}. If $ C_{ \varphi} $ is bounded on $ \HH(b)$, then the function
\begin{equation}\label{defn-u} 
u = \frac{(a_1 \circ \varphi ) \prod_{j = p + 1}^n ( \varphi - \varphi ( \xi_j))^{m_j}}{a_1}
\end{equation}
belongs to $H^2$. 
\end{Lemma}

\begin{proof} 
Denote by 
\[
v=(a_1\circ\varphi)\cdot \prod_{j=p+1}^n (\varphi-\varphi(\xi_j))^{m_j},
\]
so that $u=v/a_1$. The proof is similar to the proof of Lemma~\ref{lemme 3.9}, with some additional difficulties due to the fact that we cannot decompose $v$ as in \eqref{eq:forme-h-lemme3.10} for the indices $i$ such that $1\leq i\leq p$.

Since $\varphi\in\HH(b)\cap H^\infty= \mathfrak M(\HH(b))$, the function $v$ belongs to $\HH(b)$. As in the proof of Lemma~\ref{lemme 3.9}, it is sufficient to prove that 
\begin{equation}\label{eq 4.10} 
v^{(k)} ( \xi_i) = 0, \qquad \mbox{for every }1 \leq i \leq n \hspace{0.1 cm}\mbox{and } 0 \leq k \leq m_i - 1.
\end{equation}
We decompose the proof of \eqref{eq 4.10} into two cases depending $1\leq i\leq p$ or $p+1\leq i\leq n$. 

First, let $1\leq i\leq p$. According to \eqref{eqa}, we have $a_1(\varphi(\xi_i))=0$, whence $v(\xi_i)=0$. In order to prove that $v^{(\ell)} ( \xi_i) = 0$ for every $1\leq \ell\leq m_i-1$, we use a similar argument used in the proof of Theorem~\ref{lemme 4.0.4.0}. We construct a function $f\in H^\infty$ such that the membership of $v\cdot (f\circ\varphi)$ to $\HH(b)$ will force the derivatives of $v$ at $\xi_i$ to vanish up to order $m_i-1$. So, first apply Lemma~\ref{lemme 4.2.1} to construct a sequence $(r_k)_k\subset (0,1)$, $r_k\to 1$ as $k\to\infty$, and such that if $z_k=\varphi(r_k\xi_i)$, $k\geq 1$, then the sequence $(z_k)_k$ is an interpolating sequence for $H^\infty$. Thus, there exists $f\in H^\infty$ such that 
$$f(z_k)= f( \varphi (r_k \xi_j))   = \left\{
    \begin{array}{ll}
        1 & \mbox{if } k \hspace{0.1 cm}  \text{is even} \\
        0 & \mbox{if } k \hspace{0.1 cm}  \text{is odd}
    \end{array} .
\right.$$
Define $h(z):=a_1(z)\prod_{j=p+1}^n(z-\varphi(\xi_j))^{m_j}f(z)$, $z\in\D$, and observe that $h\circ\varphi=v\cdot(f\circ \varphi)$. Since $h\in a_1 H^2\subset\HH(b)$ and since $C_\varphi$ is bounded on $\HH(b)$, the function $h\circ\varphi\in\HH(b)$. In particular, according to \eqref{q}, for every $0\leq \ell\leq m_i-1$, the function $(h\circ\varphi)^{(\ell)}$ has a non-tangential limit at $\xi_i$. It follows from \eqref{eq:taylor} that we can write
\[
(h\circ\varphi)(z)=\sum_{\ell=0}^{m_i-1}\frac{(h\circ\varphi)^{(\ell)}(\xi_i)}{\ell!}(z-\xi_i)^\ell+(z-\xi_i)^{m_i-1}\varepsilon(z),
\]
where $\varepsilon$ is an analytic function on $\D$ which has a zero non-tangential limit at point $\xi_i$. Observe that 
\[
|(h \circ \varphi)(r \xi_i)| \leq |v ( r \xi_i )| \hspace{0.1 cm} \|f\|_{ \infty} \longrightarrow 0, \quad \mbox{as }r\to 1^-.
\]
Hence $( h \circ \varphi )( \xi_i) = 0$. Then we deduce from Taylor's formula above that
\begin{align*}
(h \circ \varphi)'(\xi_i) &= \lim\limits_{k \rightarrow \infty} \frac{(h \circ \varphi)(r_k \xi_i)-(h\circ\varphi)(\xi_i)}{(r_k - 1) \xi_i}\\
&= \lim\limits_{k \rightarrow \infty} \left(\frac{v(r_k \xi_i)-v(\xi_i)}{(r_k - 1) \xi_i} . f(z_k)  \right).
\end{align*}
Since $ \lim\limits_{k \rightarrow \infty} \dfrac{v(r_k \xi_i)}{(r_k - 1) \xi_i} = v' ( \xi_i)$, and $f(z_k) = 1$ if $k$ is even and $0$ if $k$ is odd, it follows necessarily that we should have $v' ( \xi_i ) = 0 = ( h \circ \varphi )' ( \xi_i)$.\\
Now using an induction argument, we easily prove that for every $0 \leq  \ell \leq m_i - 1 ,$
$$ v^{(\ell)} ( \xi_i) = (h \circ \varphi )^{(\ell)} ( \xi_i) = 0,$$
which concludes the proof of \eqref{eq 4.10} when $1\leq i\leq p$.

Second, let $p+1\leq i\leq n$. Observe that in this case, we have $v=(\varphi-\varphi(\xi_i))^{m_i}\psi_i$, where $\psi_i$ is an analytic function which has a non-tangential limit at $\xi_i$. Then, argue as in Lemma~\ref{lemme 3.9}, to get that $v^{(\ell)}(\xi_i)=0$ for every $0\leq\ell\leq m_i-1$. 
\end{proof}

\begin{Remark}
Compared to Lemma~\ref{lemme 3.9}, the conclusion of Lemma~\ref{lemme 0.4.0.17} is no longer true if we replace the assumption that $C_\varphi$ is bounded on $\HH(b)$ by the weaker assumption that $\varphi\in\HH(b)$. See Example~\ref{example-bounded-or-not-8} (c) where we exhibit a function $\varphi$ belonging to $\HH(b)\cap \operatorname{ball}(H^{\infty})$ but the corresponding $u$ is not in $H^2$.  
\end{Remark}

\section{Some characterization of boundedness}\label{sec4}

Recall that $b$ is assumed to be a rational function (but not a finite Blaschke product) in $\operatorname{ball}(H^{\infty})$, $a$ is its pythagorean mate to which we associate the polynomial $a_1(z)=\prod_{j=1}^n (z-\xi_j)^{m_j}$ as in \eqref{eq:definition of a}.  
\subsection{The link with weighted composition operators on $H^2$}
The following result is the key point to characterize the boundedness of composition operators on $\HH(b)$, making a connection with the boundedness of some weighted composition operators on $H^2$. 

\begin{Theorem}\label{thm 0.4.0.20}
Let $ \varphi\in\HH(b)\cap \operatorname{ball}(H^{\infty})$ and assume that $\varphi$ satisfies \eqref{eqa} and \eqref{eqb}. Let $u$ be the function defined by \eqref{defn-u}. Then the following assertions are equivalent: 
\begin{enumerate}
\item[$(i)$] $C_{\varphi}: \HH(b) \rightarrow \HH(b)$ is bounded;
\item[$(ii)$] $W_{u, \varphi}: H^2 \rightarrow H^2$ is bounded, where $W_{u,\varphi}(f)=u (f\circ\varphi)$, $f\in H^2$.
\end{enumerate}
\end{Theorem}

\begin{proof}
Denote by $\lambda_j=\varphi(\xi_j)$ for $p+1\leq j\leq n$. 

$(i)\implies(ii)$: assume that $C_{ \varphi}$ is bounded on  $ \HH(b)$, and let $f \in H^2$. Consider the function $h$ defined by
\[
h(z):= a_1 (z) \prod_{j=p+1}^n (z- \lambda_j)^{m_j} f(z),\qquad z\in\D.
\]
Since $h \in a_1 H^2 \subset \HH(b)$, the function $h \circ \varphi \in \HH(b)$. According to \eqref{eq:formula for H(b)}, decompose $h\circ\varphi$ as $h \circ \varphi = a_1 g + p$,  with $g \in H^2$ and $p \in \P_{N-1}$. Rewrite this as

\begin{equation}\label{eq 4.12} 
\frac{h \circ \varphi}{a_1} = g + \frac{p}{a_1},
\end{equation}
and observe that
\[
\frac{h \circ \varphi}{a_1}= \frac{(a_1 \circ \varphi) \prod_{j=p+1}^n ( \varphi - \lambda_j)^{m_j}}{a_1} (f \circ \varphi)= u ( f \circ \varphi ) .
\]
Since $C_{ \varphi}$ is bounded on $ \HH(b)$, Lemma \ref{lemme 0.4.0.17} implies that $u \in H^2$ and by the Littlewood subordination principle, $f \circ \varphi \in H^2$. Thus $u. ( f \circ \varphi ) \in H^1$ and \eqref{eq 4.12} implies that $p/a_1\in H^1$. Using the fact that $\mbox{deg}(p)\leq N-1<\mbox{deg}(a_1)$, it is not difficult to see that necessarily $p=0$. Therefore $u.(f \circ \varphi ) = \frac{h \circ \varphi}{a_1} = g\in H^2$ and  $W_{u, \varphi} (H^2) \subset H^2$. It remains to apply the closed graph theorem to get that $W_{u, \varphi}$ is bounded on $H^2$.

$(ii) \implies (i)$: assume now that $W_{u, \varphi}$ is bounded on $H^2$. According to Lemma~\ref{lemme 4.2.1.1}, it is sufficient to prove that for every $f\in\HH(b)$, we have $f\circ\varphi\in\HH(b)$. Using \eqref{eq:formula for H(b)}, this is equivalent to prove that 
\begin{equation}\label{eq1:composition-general833}
p\circ\varphi\in\HH(b)\qquad \mbox{for every polynomial }p\in\P_{N-1},
\end{equation}
and 
\begin{equation}\label{eq2:composition-general833}
(a_1\circ\varphi)(g\circ\varphi)\in\HH(b)\qquad \mbox{for every function }g\in H^2.
\end{equation}

For the property \eqref{eq1:composition-general833}, observe that according to \eqref{eq:multiplier}, $\varphi\in H^\infty\cap\HH(b)=\mathfrak M(\HH(b))$, and since $1\in\HH(b)$, it follows that $\varphi^n\in\HH(b)$ for every $n\geq 0$. Hence for every polynomial $p\in\P_{N-1}$, the function $p\circ\varphi$ belongs to $\HH(b)$.

For the property \eqref{eq2:composition-general833}, let us consider the finite Blaschke product associated to the sequence $ \lambda_{p+1}, \ldots , \lambda_n$ (with multiplicities $m_{p+1}, \ldots,m_n$). In other words, 
\[
B(z) = \prod_{j=p+1}^n \left(\frac{z - \lambda_j}{1- \bar{\lambda}_j z}  \right)^{m_j}, \qquad z \in \mathbb{D}.
\]
Recall that we have an explicit description of $K_B=(BH^2)^\perp$ given by \eqref{descrip-model-space-FBP}. In order to prove \eqref{eq2:composition-general833}, using that $H^2=BH^2\oplus K_B$, we will also decompose the proof into two steps: 
\begin{equation}\label{eq3:composition-general833}
(a_1 \circ \varphi ) ((Bh) \circ \varphi) \in \HH(b),\qquad \mbox{for every function } h\in H^2,
\end{equation}
and 
\begin{equation}\label{eq4:composition-general833}
( a_1 \circ \varphi) \left( \frac{p \circ \varphi}{\prod_{j=p+1}^n (1 - \bar{ \lambda}_j \varphi)^{m_j}}   \right) \in \HH(b),\qquad \mbox{for every polynomial } p\in \mathcal{P}_{N_1-1},
\end{equation}
where $N_1=\sum_{j=p+1}^n m_j$. For \eqref{eq3:composition-general833}, observe that 
\begin{align*}
(a_1 \circ \varphi) \left( (Bh)\circ \varphi \right) &= (a_1 \circ \varphi ) (B \circ \varphi)(h \circ \varphi ) \\
&= a_1 u \frac{1}{\prod_{j=p+1}^n (1 - \bar{\lambda}_j \varphi )^{m_j}} (h \circ \varphi)\\
&= a_1 W_{u, \varphi} (h_1),
\end{align*}
where $h_1 = h/\prod_{j=p+1}^n (1 - \bar{ \lambda}_j z)^{m_j}$. Since $h_1\in H^2$ and $W_{u, \varphi}$ is bounded on $H^2$, we deduce that the function
$(a_1 \circ \varphi) \left((Bh) \circ \varphi \right)$ belongs to $a_1 H^2 \subset \HH(b)$, which proves \eqref{eq3:composition-general833}. 

For \eqref{eq4:composition-general833}, observe that according to Corollary~\ref{cor 0.4.0.19}, for $p+1\leq j\leq n$, the function $(1-\overline{\lambda_j}\varphi)^{-1}\in \mathfrak M(\HH(b))$. Moreover, since $\varphi$ is also a multiplier of $\HH(b)$ and the set $\mathfrak M(\HH(b))$ is an algebra, it follows that 
 \[
\frac{(a_1 \circ \varphi).(p \circ \varphi)}{\prod_{j=p+1}^n (1- \bar{\lambda}_j \varphi)^{m_j}} \in  \mathfrak M(\HH(b))\subset \HH(b),
\]
which proves \eqref{eq4:composition-general833}. Therefore \eqref{eq2:composition-general833} is also satisfied and then $C_{ \varphi}$ is bounded on $ \HH(b)$. 
\end{proof}

It turns out that the boundedness of weighted composition operators on $H^2$ has already been characterized by M. Contreras and A. Hern\'{a}ndez-D\'{\i}az in \cite{MR1864316}. Let us introduce the Borel measure $\mu_{u,\varphi}$ on $\overline{\D}$ defined by
\begin{equation}\label{eq:mesure-image-pondere}
\mu_{u,\varphi}(E)=\int_{\varphi^{-1}(E)\cap\T}|u|^2\,dm
\end{equation}
for every measurable subset $E$ of the closed unit disk $\overline\D$. Recall that a finite Borel measure $\mu$ on  $\overline{\D}$ is called a {\emph{Carleson measure}} if there exists a constant $M>0$ such that 
\[
\mu(S(\xi,r))\leq M r,
\]
for any Carleson window $S(\xi,r)=\{z\in\overline{\D}:|z-\zeta|\leq r\}$, where $\xi\in\T$ and $0<r<1$. It is proved in \cite{MR1864316} that $W_{u,\varphi}$ is bounded on $H^2$ if and only if $\mu_{u,\varphi}$ is a Carleson measure. Moreover, recall that the Carleson property for a measure $\mu$ is equivalent to the embedding $H^2\subset L^2(\mu)$ (as was proved by L. Carleson \cite{MR141789}) and this embedding satisfies {\emph{the reproducing kernel thesis}}, meaning that $H^2\subset L^2(\mu)$ is equivalent to 
\[
\sup_{w\in\D}\int_{\overline{\D}}\frac{1-|w|^2}{|1-\overline{w}\xi|^2}\,d\mu(\xi)=\sup_{w\in\D}\frac{\|k_w\|^2_{L^2(\mu)}}{\|k_w\|_2^2}<\infty.
\]
See \cite[Lecture V11]{MR827223} or \cite[Theorem 5.15]{MR3497010}. See also \cite{MR2672342} for a discussion on the boundedness of weighted composition operators on $H^2$.
 
Using Theorem~\ref{thm 0.4.0.20}, we therefore immediately get from this the following characterization for the boundedness of $C_\varphi$ on $\HH(b)$, when $b$ is a rational function (but not a finite Blaschke product) in $\operatorname{ball}(H^{\infty})$.
\begin{Corollary}\label{Cor 0.4.0.20}
Let $ \varphi\in\HH(b)\cap \operatorname{ball}(H^{\infty})$ and assume that $\varphi$ satisfies \eqref{eqa} and \eqref{eqb}. Let $u$ be the function defined by \eqref{defn-u} and let $\mu_{u,\varphi}$ be the measure defined by \eqref{eq:mesure-image-pondere}. Then the following assertions are equivalent: 
\begin{enumerate}
\item[$(i)$] $C_{\varphi}: \HH(b) \rightarrow \HH(b)$ is bounded;
\item[$(ii)$] $\mu_{u, \varphi}$ is a Carleson measure;
\item[$(iii)$] $\sup_{w\in\D}\displaystyle{  \int_\mathbb{T}} (1 - |w|^2 ) \frac{|u( \xi)|^2}{|1 - \overline{w} \varphi ( \xi)|^2}\,dm(\xi) < + \infty$.
\end{enumerate}
\end{Corollary}

\begin{proof}
The equivalence $(i)\Longleftrightarrow (ii)$ follows from Theorem~\ref{thm 0.4.0.20} and the result of M. Contreras and A. Hern\'{a}ndez-D\'{\i}az mentioned above. 
The equivalence $(ii)\Longleftrightarrow (iii)$ has been already observed in \cite{MR2672342}: it follows from the fact that $\mu_{u, \varphi}$ is a Carleson measure if and only if 
\[
\sup_{w\in\D}\int_{\overline{\D}}\frac{1-|w|^2}{|1-\overline{w}\xi|^2}\,d\mu_{u,\varphi}(\xi)<\infty,
\]
which, by a change of variable, is equivalent to $(iii)$. 
\end{proof}

\begin{Example}\label{example-bounded-or-not-8}
Let $a(z)=c(z-1)(z+1)^2$, $z\in\D$, where $c$ is some suitable constant such that $a\in\operatorname{ball}(H^{\infty})$ and let $b$ be its pythagorean mate (remind that $b$ can be constructed using the Fej\'er--Riesz Theorem). In the following examples, we consider symbols $\varphi\in\operatorname{ball}(H^{\infty})$ which are analytic in a neighborhood of $\overline{\D}$. Thus, they belong to $\HH(b)$ (because $b$ is a non-extreme point of $\operatorname{ball}(H^{\infty})$).

\begin{enumerate}
\item Let $\varphi(z)=(z+1)/2$, $z\in\D$. Then $\varphi(1)=1$, $\varphi(-1)=0$, and it can be checked that $u(z)=(z+3)^2/32$. Hence $u\in H^\infty$ and by the Littlewood subordination principle, the operator $W_{u,\varphi}$ is bounded on $H^2$, which implies by Theorem~\ref{thm 0.4.0.20} that the operator $C_\varphi$ is bounded on $\HH(b)$. 
%\item Let $\varphi(z)=(1-z)/2$, $z\in\D$. Then $\varphi\in\HH(b)$, $\varphi(1)=0$ and $\varphi(-1)=1$ and it can be easily checked that $u(z)=(z+3)/16$. Hence $u\in H^\infty$ and then by the Littlewood subordination principle, the operator $W_{u,\varphi}$ is bounded on $H^2$. Therefore, Theorem~\ref{thm 0.4.0.20} implies that $C_\varphi$ is bounded on $\HH(b)$. 
%\item Let $\varphi(z)=(z+1)/4$, $z\in\D$. Then $\varphi\in\HH(b)$, $\varphi(1)=1/2$ and $\varphi(-1)=0$ and it can be easily checked that $u(z)=(z-3)(z+5)^2/16^3$. Hence $u\in H^\infty$ and then by the Littlewood subordination principle, the operator $W_{u,\varphi}$ is bounded on $H^2$. Therefore, Theorem~\ref{thm 0.4.0.20} implies that $C_\varphi$ is bounded on $\HH(b)$. 
\item Let $\varphi(z)=(z-r)/(1-rz)$, $r\in (0,1)$. Then $\varphi(1)=1$, $\varphi(-1)=-1$, and it can be checked that $u(z)=(1-r^2)(1-r)/(1-rz)^3$. Hence $u\in H^\infty$ and by the Littlewood subordination principle, the operator $W_{u,\varphi}$ is bounded on $H^2$. Therefore, according to Theorem~\ref{thm 0.4.0.20}, the operator  $C_\varphi$ is bounded on $\HH(b)$. 
\item Let $\varphi(z)=z^2$, $z\in\D$. Then $\varphi(1)=\varphi(-1)=1$ and it can be checked that $u(z)=(z^2+1)^2/(z+1)$. Hence $u\notin H^2$, and it follows from Lemma~\ref{lemme 0.4.0.17} that the operator $C_\varphi$ is not bounded on $\HH(b)$. 
%\item Let $\varphi(z)=-z^2$, $z\in\D$. Then $\varphi\in\HH(b)$, $\varphi(1)=\varphi(-1)=-1$ and it can be easily checked that $u(z)=-(z^2+1)(z-1)$. Hence $u\in H^\infty$. In particular, the operator $W_{u,\varphi}$ is bounded on $H^2$ and then $C_\varphi$ is bounded on $\HH(b)$. 
\end{enumerate}
\end{Example}

In Examples \ref{example-bounded-or-not-8} (a) and (b), the associated function $u\in H^\infty$ and so the boundedness of $W_{u,\varphi}$ on $H^2$ is a simple consequence of the Littlewood subordination principle. We will exhibit, in Example~\ref{example-blaschke-u-non-borne} below, a rational function $b$ and a symbol $\varphi$ for which the associated function $u\in H^2\setminus H^\infty$ but yet generate a bounded operator $C_\varphi$ on $\HH(b)$, and thus a bounded operator $W_{u,\varphi}$ on $H^2$.

Using results from \cite{MR2672342}, we can also obtain an interesting sufficient condition for the boundedness.
\begin{Corollary}\label{Corollary-Gallardo-Partington}
Let $ \varphi\in\HH(b)\cap \operatorname{ball}(H^{\infty})$ and assume that $\varphi$ satisfies \eqref{eqa} and \eqref{eqb}. Let $u$ be the function defined by \eqref{defn-u} and assume that $u\in H^2$. Assume also that for some $\delta>0$, we have
\begin{equation}\label{eq:equation-c-delta}
\supess_{z\in A_\delta}|u(z)|<\infty,
\end{equation}
where $A_\delta:=\{\zeta\in\mathbb T:|\varphi(z)|\geq 1-\delta\}$. Then the operator $C_\varphi$ is bounded on $\HH(b)$. 
\end{Corollary}
\begin{proof}
According to Theorem~\ref{thm 0.4.0.20}, the operator $C_\varphi$ is bounded on $\HH(b)$ if and only if the operator $W_{u,\varphi}$ is bounded on $H^2$. Now the conclusion follows applying \cite[Theorem 2.9]{MR2672342}.
\end{proof}

In the rest of this section, we give some additional necessary and/or sufficient conditions more tractable for the boundedness of the composition operator $C_\varphi$ on $\HH(b)$. 

\subsection{Angular derivatives and boundedness}
We now explain how the existence of angular derivative in the sense of Carath\'eodory for the symbol $\varphi$ is involved in the boundedness of $C_\varphi$. In particular, we extend a result of D. Sarason and J.N. Silva in our context. See \cite[Theorem 2]{MR1945292}. Recall first that a function $\varphi\in\operatorname{ball}(H^{\infty})$ has an {\emph{angular derivative in the sense of Carath\'eodory} (briefly an ADC) at a point $\xi\in\T$ if $\varphi$ and $\varphi'$ both have a non-tangential limit at $\xi$ and $|\varphi(\xi)|=1$. A well-known characterization of Carath\'eodory says that $\varphi$ has an ADC at $\xi\in\T$ if and only if 
\[
c:=\liminf_{z\to\xi}\frac{1-|\varphi(z)|}{1-|z|}<\infty.
\]
Moreover in this case, $c=|\varphi'(\xi)|>0$. See for instance \cite[Theorem 2.44]{MR1397026}. D. Sarason also observed that if $\varphi$ has an ADC at $\xi$, then
\begin{equation}\label{Sarason-ADC343535}
z\longmapsto \frac{\varphi(z)-\varphi(\xi)}{z-\xi} \mbox{ is in }\HH(\varphi).
\end{equation}
 See \cite{MR946094} or \cite[Theorem 21.1]{MR3617311}.

\begin{Corollary}\label{cor 0.4.0.21}
Let $ \varphi\in\HH(b)\cap \operatorname{ball}(H^{\infty})$ and assume that $\varphi$ satisfies  \eqref{eqa} and \eqref{eqb}. For $1\leq k\leq p$, let $1\leq \ell_k\leq n$ such that $\varphi(\xi_k)=\xi_{\ell_k}$. Assume that the operator $C_\varphi$ is bounded on $\HH(b)$. Then for every $1\leq k\leq p$, we have $m_k\leq m_{\ell_k}$. Furthermore, if for some $1\leq k\leq p$, we have $m_k=m_{\ell_k}$, then the function $\varphi$ has an ADC at $\xi_k$. 
\end{Corollary}
\begin{proof}
It follows from our assumption and Theorem~\ref{thm 0.4.0.20} that the operator $W_{u,\varphi}$ is bounded on $H^2$, where $u$ is defined by \eqref{defn-u}. Hence, there is a constant $C>0$ such that, for every $g\in H^2$, we have
\[
\|W_{u , \varphi}^* g\|_2 \leq C \|g\|_2 . 
\]
Apply this inequality with $g=k_\lambda$, $\lambda\in\D$, where $k_\lambda(z)=(1-\overline{\lambda}z)^{-1}$, and use the well-known fact that $W_{u,\varphi}^*k_\lambda=\overline{u(\lambda})k_{\varphi(\lambda)}$ to get 
\begin{equation}\label{eq1:cor-04-0-21}
 |u ( \lambda)| \| k_{\varphi (  \lambda)}\|_2 \leq C \|k_{ \lambda}\|_2. 
 \end{equation}
Let's fix $1\leq k\leq p$, and for simplicity, denote by $\ell=\ell_k$, so that $\varphi(\xi_k)=\xi_\ell$. Applying \eqref{eq1:cor-04-0-21} with $\lambda=r\xi_k$, for $0<r<1$, gives
\[
 \frac{1 - r^2}{1 - | \varphi ( r \xi_k )|^2}|u( r \xi_k )|^2  \leq C^2 .
\]
With the definition of $u$, we obtain
\[
\frac{1 - r^2}{1 - | \varphi ( r \xi_k)|^2} \prod_{i=1}^n \left| \frac{ \varphi(r \xi_k ) - \xi_i }{r \xi_k - \xi_i} \right|^{2 m_i} \prod_{j = p+1}^n | \varphi (r \xi_k) - \lambda_j|^{2 m_j}  \leq C^2,
\]
where $ \lambda_j = \varphi ( \xi_j) \in \mathbb{D}$, for $p+1\leq j\leq n$. The last inequality can be rewritten as 
\[
\frac{1-r^2}{1 - | \varphi ( r \xi_k)|^2} \frac{ | \varphi(r \xi_k) - \xi_l |^{2 m_l}}{|r \xi_k - \xi_k|^{2 m_k}} \frac{\kappa_{\ell,r}}{c_{k,r}}\prod_{j=p+1}^n |\varphi(r\xi_k)-\lambda_j|^{2m_j}\leq C^2,
\]
where 
\[
\kappa_{\ell,r}=\prod_{i=1,i\neq\ell}^n|\varphi(r\xi_k)-\xi_i|^{2m_i}\quad\mbox{and}\quad c_{k,r}=\prod_{i=1,i\neq k}^n|r\xi_k-\xi_i|^{2m_i}.
\]
But, using that $ |\varphi (r \xi_k) - \xi_{ l}| \geq 1 - | \varphi ( r \xi_k)|$, we have 
\begin{align*}
\frac{1-r^2}{1 - | \varphi ( r \xi_k)|^2} . \frac{| \varphi (r \xi_k) - \xi_l|^{2 m_l}}{|r \xi_k - \xi_k|^{2 m_k}}
 &\geq \frac{(1-r)(1-|\varphi(r \xi_k )|)^{2 m_l}}{(1 - |\varphi ( r \xi_k)|) (1+| \varphi (r \xi_k)|)(1 -r)^{2m_k}}\\
 & \geq \frac{1}{2} \frac{(1 - |\varphi(r \xi_k)|)^{2 m_l - 1}}{(1-r)^{2m_k - 1}},
\end{align*} 
from which it follows that
 \[ 
 \frac{ ( 1 - | \varphi(r \xi_k)  |)^{2 m_l - 1}}{(1  - r)^{2 m_k - 1}} \leq 2  C^2  \frac{c_{k,r}}{\kappa_{\ell,r} \prod_{j = p+1}^n | \varphi (r \xi_k) - \lambda_j|^{2 m_j}}.
 \]
This can be rewritten as 
\begin{equation}\label{eq:ADC-composition-borne-48343}
\left( \frac{  1 - | \varphi(r \xi_k)  |}{1  - r} \right)^{2 m_l - 1}  \leq 2  C^2  \frac{c_{k,r}(1-r)^{2m_k-2m_{\ell}}}{\kappa_{\ell,r} \prod_{j = p+1}^n | \varphi (r \xi_k) - \lambda_j|^{2 m_j}}.
 \end{equation}
 Observe now that 
\[
\frac{c_{k,r}}{\kappa_{\ell,r}\prod_{j=p+1}^n|\varphi(r\xi_k)-\lambda_j|^{2m_j}}\to C_{k,\ell},\quad\mbox{as }r\to 1,
\]
where
\[
C_{k,\ell}:=\frac{ \prod_{i=1, i \neq k}^n |\xi_k - \xi_i |^{2 m_i} }{ \prod_{i = 1, i \neq l}^n|   \xi_l - \xi_i |^{2 m_i}} \frac{1}{ \prod_{j = p+1}^n | \xi_l - \lambda_j|^{2 m_j}}<\infty.
\]
If $m_k>m_\ell$, then it follows from \eqref{eq:ADC-composition-borne-48343} that
\[
\liminf_{r \rightarrow 1} \left( \frac{  1 - | \varphi(r \xi_k)  |}{1  - r} \right)^{2 m_l - 1}=0,
\]
which is not possible by Carath\'eodory's theorem. Therefore, $m_k\leq m_\ell$. Moreover, if $m_k=m_\ell$, then it also follows from \eqref{eq:ADC-composition-borne-48343} that 
\[
\liminf_{r \rightarrow 1}  \frac{  1 - | \varphi(r \xi_k)  |}{1  - r} < \infty,
\]
and by Carath\'eodory's theorem, the function $ \varphi$ has an ADC at point $ \xi_k$. 
 \end{proof}

\begin{Remark}
Let's come back to Examples~\ref{example-bounded-or-not-8}. In these examples, $a$ has two zeros $\zeta_1=1$ and $\zeta_2=-1$, with multiplicities $m_1=1$ and $m_2=2$. If $\varphi\in\HH(b)\cap \operatorname{ball}(H^{\infty})$ with $\varphi(-1)=1$, then, according to Corollary~\ref{cor 0.4.0.21}, the operator $C_\varphi$ cannot be bounded on $\HH(b)$. In other words, if $C_\varphi$ is bounded on $\HH(b)$, then either $\varphi(-1)\in\D$ or $\varphi(-1)=-1$. In particular, we recover the conclusion of Example~\ref{example-bounded-or-not-8} (c).
\end{Remark}

The following result is a partial converse of Lemma~\ref{lemme 0.4.0.17}.
\begin{Corollary}\label{cor-u-dans-Hvarphi-Cvarphi-bdd}
Let $ \varphi\in\HH(b)\cap \operatorname{ball}(H^{\infty})$ which satisfies \eqref{eqa} and \eqref{eqb}. Let $u$ be the function defined by \eqref{defn-u}. Assume that $u\in\HH(\varphi)$. Then the operator $C_\varphi$ is bounded on $\HH(b)$.
\end{Corollary}

\begin{proof}
Since $u\in\HH(\varphi)$, it follows from a result of M. Jury \cite{MR2336583} that the operator $A= C_{ \varphi}^* T_u^*$ is bounded from $H^2$ into itself, where $T_u^*$ is viewed here as an unbounded operator, densely defined on reproducing kernels of $H^2$ by the formula 
\[
 T_u^* k_{ \lambda} = \overline{u( \lambda)} k_{\lambda} , \hspace{0.2 cm} \lambda \in \mathbb{D}.
 \]
Hence the operator $A$ has an adjoint $A^* : H^2 \rightarrow H^2 $ which is bounded and satisfies for every $f\in H^2$ and every $\lambda\in\D$ 
\begin{align*}
(A^*f)( \lambda) &= < A^* f, k_{ \lambda}>_2 = <f, A k_{ \lambda}>_2 \\
&= <f, C_{ \varphi}^* T_u^* k_{ \lambda}>_2\\
&= <C_{ \varphi} f, \overline{u( \lambda)} k_{ \lambda}>_2\\
&= u( \lambda) f ( \varphi ( \lambda )).
\end{align*}
For every $f \in H^2$ and $ \lambda \in \mathbb{D}$, we then have 
\[
(W_{u, \varphi} f) ( \lambda) = (A^* f)( \lambda).
\]
Finally, $W_{u, \varphi}=A^*$ is bounded on $H^2$, which implies, according to Theorem~\ref{thm 0.4.0.20}, that $C_{ \varphi}$ is bounded on $ \HH(b)$. 
\end{proof}

\begin{Remark}
It should be noted that the condition $u\in\HH(\varphi)$ is not, in general, a necessary condition for the boundedness of the operator $C_\varphi$ on $\HH(b)$. Indeed, if we come back to Example~\ref{example-bounded-or-not-8} (b), the symbol $\varphi$ is a Blaschke factor associated to point $r\in (0,1)$, and then $\HH(\varphi)=K_\varphi=\mathbb C k_r$, where $k_r(z)=(1-rz)^{-1}$. Moreover, the operator $C_\varphi$ is bounded on $\HH(b)$, whereas the associated function $u$ is not a constant multiple of $k_r$, and thus $u\notin \HH(\varphi)$. 
\end{Remark}
Nevertheless, the next result gives a particular situation where the converse of Corollary~\ref{cor-u-dans-Hvarphi-Cvarphi-bdd}, as well as the converse of Corollary~\ref{cor 0.4.0.21} on angular derivatives, hold true.

\begin{Corollary}\label{cor 004}
Let $b(z)=(1+z)/2$, $z\in\D$, and let $ \varphi \in \HH(b)\cap \operatorname{ball}(H^{\infty})$. Assume that $ \varphi(1)=1$ and let $u$ be the function defined by \eqref{defn-u}. Then the following assertions are equivalent:
\begin{enumerate}
\item[$(i)$] the operator $C_{ \varphi}: \HH(b) \rightarrow \HH(b)$ is bounded;
\item[$(ii)$] the function $\varphi$ has an ADC at point $1$;
\item[$(iii)$] the function $u$ belongs to $\HH(\varphi)$. 
\end{enumerate}
\end{Corollary}
\begin{proof}
Observe that when $b(z)=(1+z)/2$, then $a_1(z)=z-1$ and since $\varphi(1)=1$, we have $u(z)=(\varphi(z)-1)/(z-1)$, $z\in\D$. The implication $(i)\implies (ii)$ follows directly from Corollary~\ref{cor 0.4.0.21}. The implication $(ii)\implies (iii)$ follows from \eqref{Sarason-ADC343535}, whereas the implication $(iii)\implies (i)$ follows from Corollary~\ref{cor-u-dans-Hvarphi-Cvarphi-bdd}.
\end{proof}

\begin{Remark}
As we already mentioned, in the case when $b(z)=(1+z)/2$, then $\HH(b)$ coincides with the local Dirichlet space $\mathcal D(\delta_1)$ associated to the Dirac measure $\delta_1$ at point $1$. In particular, we may translate Corollary~\ref{cor 004} in the context of $\mathcal D(\delta_1)$ and recover a result of D. Sarason and J.N. Silva \cite[Theorem 2]{MR1945292}. However, our proof is different from Sarason and Silva's proof. More precisely, they use a characterization of boundedness which is based on some counting function adapted to the situation of $\mathcal D(\delta_1)$ space, whereas we use a characterization based on the boundedness of some related weighted composition operator on $H^2$.
\end{Remark}

\begin{Example}\label{example-blaschke-u-non-borne}
Let $b(z)=(z+1)/2$ and let $(\lambda_n)_{n \geq 1}$ be a sequence in $\mathbb{D}$ tending to $1$ and satisfying
\begin{equation}\label{eq 3.15.1}
\sum_{n \geq 1} \frac{1-|\lambda_n|}{|\lambda_n - 1|^2} < \infty.
\end{equation}
(For example, we can take $\lambda_n = (1-2^{-n})\exp(\frac{i}{n})$, $n\geq 1$). Let us now consider the Blaschke product $B$ associated to $(\lambda_n)_n$. A result of  O. Frostman \cite{MR12127} says that, under the condition \eqref{eq 3.15.1}, the function $B$ has an ADC at point $1$. Multiplying $B$ by a unimodular constant if necessary, we can assume that $B(1)=1$. The property \eqref{Sarason-ADC343535} now implies that the function
 $$u(z) = \frac{B(z) - 1}{z-1},\qquad z\in\D,$$ 
belongs to $K_B$ and then, in particular to $H^2$. Since $B(z) = (z-1)u(z) +1$ and $u \in H^2$, then $B \in \HH(b)$. According to Corollary~\ref{cor 004},  the operator $C_B$ is bounded on $ \HH(b)$. Observe that in that case $u\notin H^\infty$  because
\[
|u(\lambda_n)| = \left|  \frac{-1}{\lambda_n - 1}\right| \longrightarrow \infty , \hspace{0.2 cm} n \rightarrow \infty.
\]
\end{Example}

\subsection{The case when $\varphi(\xi_j)\in\D$, $1\leq j\leq n$.} We have given a necessary condition on the behavior of $\varphi$ at a point $\xi_k\in Z_\T(a)$ when $\varphi(\xi_k)\in\T$. In the other direction, when $\varphi(\xi_j)\in\D$, $1\leq j\leq n$, we now give a sufficient condition. 

\begin{Corollary}\label{cor 0.4.0.22}
Let  $\varphi\in\HH(b)\cap \operatorname{ball}(H^{\infty})$ and assume that for every $1\leq j\leq n$, we have 
\begin{equation}\label{eq:limsupuniformement-loin-de-un-343434}
\limsup_{z\to \xi_j}|\varphi(z)|<1.
\end{equation}
Then the operator $C_{ \varphi}$ is bounded on $ \HH(b)$. 
\end{Corollary}
\begin{proof}
In order to show that $C_\varphi$ is bounded on $\HH(b)$, we will use Corollary~\ref{Cor 0.4.0.20} and prove that
\begin{equation}\label{eq 3.15}
\sup_{w \in \mathbb{D}} \int_{\mathbb{T}} (1 - |w|^2) \frac{|u ( \xi)|^2}{|1 - \overline{w} \varphi (\xi)|^2} dm (\xi) < \infty ,
\end{equation}
where $u$ is defined by
\[
u =\frac{ (a_1 \circ \varphi) \prod_{j=1}^n ( \varphi - \lambda_j)^{m_j} }{a_1} = \frac{v}{a_1} 
\]
with $v =  (a_1 \circ \varphi) \prod_{j=1}^n ( \varphi - \lambda_j)^{m_j}$. (According to \eqref{eq:limsupuniformement-loin-de-un-343434}, observe that $\lambda_j=\varphi(\xi_j)\in\D$ for every $1\leq j\leq n$, whence $p=0$ in the notation of \eqref{eqa}). It follows from our assumption that there exists $0<L<1$ and $\delta>0$ such that for every $1\leq j\leq n$ and for every $z\in\D$, we have 
\[
|z-\xi_j|<\delta\implies |\varphi(z)|\leq L.
\]
Let $V_{\xi_j}=\{\xi\in\T:|\xi-\xi_j|<\delta\}$ and $V=\bigcup_{j=1}^n V_{ \xi_j}$. Then, for almost all $\xi\in V$, we have $|\varphi(\xi)|\leq L$. 

On one hand, observe now that on $\T\setminus V$, the function $a_1$ is bounded below, and then $u=v/a_1$ is bounded by some constant $C$. Thus, for every $w \in \mathbb{D}$, we get
\begin{align*} 
\int_{ \mathbb{T} \setminus V} (1 - |w|^2 ) \frac{|u( \xi)|^2}{|1 - \overline{w} \varphi ( \xi)|^2} dm( \xi) &\leq C^2 (1 - |w|^2 ) \int_{ \mathbb{T}} \frac{1}{|1 - \overline{w} \varphi ( \xi)|^2} dm( \xi)\\
&= C^2 (1 - |w|^2 ) \| C_{ \varphi} k_w \|_2^2 
\end{align*}
Now, by the Littlewood subordination principle, the operator $C_{ \varphi}$ is bounded on $H^2$, whence 
\[
\|C_{\varphi} k_w\|_2^2 \leq \|C_{ \varphi}\|^2 \|k_w\|_2^2 = \|C_{ \varphi}\|^2 (1 - |w|^2 )^{-1} .
\]
Finally, for every $w \in \mathbb{D}$, we obtain
$$ \int_{ \mathbb{T}  \setminus V} (1 - |w|^2) \frac{|u( \xi)|^2}{|1 - \overline{w} \varphi ( \xi)|^2} dm( \xi) \leq C^2  \|C_{ \varphi}\|^2 .$$
On the other hand, for $w \in \mathbb{D}$ and for almost all $ \xi \in V$, we have 
\[
|1 - \overline{w} \varphi( \xi)| \geq 1 - |w| \hspace{0.1 cm } | \varphi( \xi)| \geq 1 - L |w| \geq 1 - L > 0.
\]
Moreover, $a_1 \circ \varphi$ is in $H^{ \infty}$ and by Lemma~\ref{lemme 3.9}, the function $\prod_{j=1}^n (\varphi(z)-\lambda_j)^{m_j}/(z-\xi_j)^{m_j}$ is in $H^2$. Hence $u \in H^2$, which gives
\[ 
\int_V \frac{(1 - |w|^2)|u(\xi)|^2}{|1 - \overline{w} \varphi (\xi)|^2} dm( \xi) \leq \frac{\|u\|_2^2}{(1- L)^2}.
\]
Finally \eqref{eq 3.15} is satisfied, from which it follows that the operator $C_{\varphi}$ is bounded on $\HH(b)$.
 \end{proof}
 
\begin{Example}\label{exmp 7} 
Let $b\in\operatorname{ball}(H^{\infty})$ be a rational function (but not a finite Blaschke product), and let $a_1(z)=\prod_{j=1}^n(z-\xi_j)^{m_j}$ be the associated polynomial as in \eqref{eq:definition of a}. Let $\psi\in H^\infty$ and let 
\[
\varphi(z)=c\prod_{j=1}^n(z-\xi_j)^{\alpha_j}\psi(z),\qquad z\in\D,
\] 
where for $1\leq j\leq n$, $\alpha_j$ is a real number satisfying $\alpha_j>m_j-1/2$, and $c$ is a constant such that $\|\varphi\|_\infty=1$. Observe that the function $\varphi:\D\longmapsto \D$ is analytic and, since $\alpha_j>1/2$, for every $1\leq j\leq n$, we have 
\[
\limsup_{z\to\xi_j}|\varphi(z)|=0<1.
\]
Moreover, since 
\[
\frac{\varphi(z)}{a_1(z)}=c\frac{\psi(z)}{\prod_{j=1}^n (z-\xi_j)^{m_j-\alpha_j}},\qquad z\in\D,
\]
and $2(m_j-\alpha_j)<1$, it is easy to check that $\varphi/a_1$ belongs to $H^2$. In particular, $\varphi\in a_1 H^2\subset\HH(b)$. Thus, we can apply Corollary~\ref{cor 0.4.0.22} which implies that $C_\varphi$ is bounded on $\HH(b)$. 
\end{Example} 

\section{Some characterization of compactness and the Hilbert--Schmidt property}\label{sec5}
As for the boundedness, we establish a link between the compactness of the operator $C_\varphi$ on $\HH(b)$ and the compactness of the operator $W_{u,\varphi}$ on $H^2$, where $u$ is defined as in \eqref{defn-u}. In this section, we still assume that $b\in\operatorname{ball}(H^{\infty})$ is a rational function (but not a finite Blaschke product), and $a$ is its pythagorean mate to which we associate the polynomial $a_1(z)=\prod_{j=1}^n(z-\xi_j)^{m_j}$ as in \eqref{eq:definition of a}.

\subsection{The link with weighted composition operator for compactness}

The following general fact will be of use to us. It is probably folklore and we leave the proof to the reader. Here $\mathcal L(\HH,\mathcal K)$ denotes the space of bounded linear operator from an Hilbert space $\HH$ into another Hilbert space $\mathcal K$.
\begin{Lemma}\label{lem-op-compact}
Let $T_1\in\mathcal L(\HH_1,\HH_2)$, $T_2\in\mathcal L(\HH_3,\HH_4)$, $V_1\in\mathcal L(\HH_3,\HH_1)$ and $V_2\in\mathcal L(\HH_4,\HH_2)$. Assume that $T_1V_1=V_2T_2$ and assume that $V_1$  is an isomorphism and $V_2$  is an isometry. Then $T_1$ is compact (respectively Hilbert--Schmidt) if and only if $T_2$ is compact (respectively Hilbert--Schmidt). 
\end{Lemma}

The following result will be the key in our characterization of compactness for the composition operators on $\HH(b)$.
\begin{Theorem}\label{thm 0.4.0.27}
Let $ \varphi\in\HH(b)\cap \operatorname{ball}(H^{\infty})$ and assume that $\varphi$ satisfies \eqref{eqa} and \eqref{eqb}. Let $u$ be the function defined by \eqref{defn-u}.  Then the following assertions are equivalent:
\begin{enumerate}
\item[(i)] $C_{\varphi}: \HH(b) \rightarrow \HH(b)$ is compact (respectively Hilbert--Schmidt).
\item[(ii)] $W_{u, \varphi}: H^2 \rightarrow H^2$ is compact (respectively Hilbert--Schmidt). 
\end{enumerate}
\end{Theorem}
\begin{proof}
According to Theorem~\ref{thm 0.4.0.20}, we can assume that both operators $C_\varphi$ and $W_{u,\varphi}$ are bounded respectively on $\HH(b)$ and $H^2$.  For $p+1\leq j\leq n$, denote by $\lambda_j := \varphi (\xi_j) \in \mathbb{D}$ and let
\[
\psi = \frac{u}{\prod_{j=p+1}^n (1 - \bar{\lambda}_j \varphi)^{m_j}} .
\]
On one hand, since $ \prod_{j=p+1}^n ( 1 - \bar{\lambda}_j \varphi )^{m_j}$ is in $H^{ \infty}$ and invertible in $H^{ \infty}$ (because $|1-\overline{\lambda_j}\varphi|\geq 1-|\lambda_j|>0$), the operator  $W_{u, \varphi}$ is bounded (respectively compact or Hilbert--Schmidt) on $H^2$ if and only if the operator $W_{ \psi , \varphi}$ is bounded (respectively compact or Hilbert--Schmidt) on $H^2$. 

On the other hand, let us consider the finite Blaschke product $B$ associated to the sequence $\lambda_{p+1}, \ldots, \lambda_{n}$, with multiplicities $m_{p+1},\ldots,m_{n}$,  and denote by $R_1$ the restriction of $C_\varphi$ to the closed subspace $a_1BH^2$ of $\HH(b)$, 
$$ \begin{array}{ccccl}
R_1&:& a_1BH^2& \longrightarrow &\HH(b) \\
 & & f & \longmapsto & R_1f  = C_{ \varphi} f = f \circ \varphi.
\end{array}
$$
According to Lemma~\ref{lemme 0.4.0.24}, the subspace $a_1BH^2$ is of finite codimension in $\HH(b)$ from which it follows that the operator $C_\varphi$ is compact (respectively Hilbert--Schmidt) on $\HH(b)$ if and only if the operator $R_1$ is compact (respectively Hilbert--Schmidt) from $a_1BH^2$ into $\HH(b)$. 

Therefore it is sufficient to prove that $R_1:a_1BH^2\rightarrow \HH(b)$ is compact (respectively Hilbert--Schmidt) if and only if $W_{ \psi , \varphi}:H^2\rightarrow H^2$ is compact (respectively Hilbert--Schmidt). To make the link between these two operators, let us introduce now
$$ \begin{array}{ccccc}
V_1&:&H^2 & \to &a_1 B H^2 \\
 & & f & \mapsto & V_1f  =a_1 B f,
\end{array}
$$
and
$$ \begin{array}{ccccc}
V_2&:&H^2 & \to & \HH(b) \\
 & & g & \mapsto & V_2g  =a_1 g.
\end{array}
$$
The subspace $a_1BH^2$ is viewed here as a closed subspace of $\HH(b)$ and in particular is equipped with the $\HH(b)$ norm $\vvvert\cdot\vvvert_b$. Thus, according to \eqref{eq:norm in h(b)}, the operators $V_1$ and $V_2$ are isometries.  Moreover, the operator $V_1$ is onto. Let us check that
\begin{equation}\label{eq 3.19}
 R_1 V_1 = V_2 W_{\psi , \varphi}.
\end{equation}  
To this purpose, take $f \in H^2$. We have 
\begin{align*}
(R_1 V_1 )(f) = R_1 (a_1 B f ) &= (a_1 \circ \varphi ) (B \circ \varphi ) (f \circ \varphi ) \\
&= a_1 \frac{(a_1 \circ \varphi )}{a_1} (B \circ \varphi ) (f \circ \varphi).
\end{align*}
But observe that
\begin{align*}
\frac{(a_1 \circ \varphi)}{a_1} (B \circ \varphi) &= \frac{(a_1 \circ \varphi )}{a_1} \prod_{j=p+1}^n \frac{( \varphi - \lambda_j)^{m_j}}{( 1 - \bar{\lambda}_j \varphi)^{m_j}}\\
&= \frac{u}{ \prod_{j = p+1}^n (1 - \bar{ \lambda}_j \varphi)^{m_j}} \\
&= \psi.
\end{align*}
Thus, for every $f \in H^2$, we have  
\begin{align*}
(R_1 V_1)(f) &= a_1 \psi (f \circ \varphi) \\
&= a_1 W_{\psi , \varphi} f \\
&= (V_2 W_{ \psi , \varphi}) (f) ,
\end{align*}
which proves \eqref{eq 3.19}. It remains now to apply Lemma~\ref{lem-op-compact} to get that the operator $R_1: a_1 B H^2 \rightarrow \HH(b)$ is compact (respectively Hilbert--Schmidt) if and only if the operator $W_{ \psi , \varphi} : H^2 \to H^2$ is compact (respectively Hilbert--Schmidt). This concludes the proof.
\end{proof}

It turns out that the compactness of weighted composition operators on $H^2$ has also been characterized by M. Contreras and A. Hern\'{a}ndez-D\'{\i}az in \cite{MR1864316}.  Recall that a finite Borel measure $\mu$ on  $\overline{\D}$ is called a {\emph{vanishing Carleson measure}} if 
\[
\lim_{r\to 0}\sup_{\xi\in\T}\frac{\mu(S(\xi,r))}{r}=0.
\]
Then, it is proved in \cite{MR1864316} that the operator $W_{u,\varphi}$ is compact on $H^2$ if and only if $\mu_{u,\varphi}$ is a vanishing Carleson measure, where $\mu_{u,\varphi}$ is defined by \eqref{eq:mesure-image-pondere}. Moreover, the vanishing Carleson property for a measure $\mu$ is equivalent to the compactness of the embedding $H^2\subset L^2(\mu)$, which satisfies the reproducing kernel thesis, meaning that $H^2$ embeds compactly into $L^2(\mu)$ if and only if 
\[
\int_{\overline{\D}}\frac{1-|w|^2}{|1-\overline{w}\xi|^2}\,d\mu(\xi)\to 0,\quad \mbox{as }|w|\to 1.
\]
See \cite{Power} for a discussion on the vanishing Carleson measures and \cite{MR2672342} for a discussion on the compactness of the weighted composition operators on $H^2$.\\

Using Theorem~\ref{thm 0.4.0.27}, we therefore immediately get from the previous discussion the following characterization for the compactness of the composition operators on $\HH(b)$.  
\begin{Corollary}\label{cor 0.4.0.27}
Let $ \varphi\in\HH(b)\cap \operatorname{ball}(H^{\infty})$  and assume that $\varphi$ satisfies \eqref{eqa} and \eqref{eqb}. Let $u$ be the function defined by \eqref{defn-u} and let $\mu_{u,\varphi}$ be the measure defined by \eqref{eq:mesure-image-pondere}. Then the following assertions are equivalent:
\begin{enumerate}
\item[(i)] $C_{\varphi}: \HH(b) \rightarrow \HH(b)$ is compact.
\item[(ii)] $ \mu_{u,  \varphi}$ is a vanishing Carleson measure.
\item[$(iii)$] $ \displaystyle{ \int_{  \mathbb{T}}} \dfrac{(1 - |w|^2) |u ( \xi )|^2}{|1 - \overline{w} \varphi ( \xi)|^2} dm( \xi) \longrightarrow 0 $ as $|w| \rightarrow 1 .$
\end{enumerate}
\end{Corollary}

We also have a characterization for the Hilbert-Schmidt property. 
\begin{Corollary}\label{thm 4.0.21.0}
Let $ \varphi\in\HH(b)\cap \operatorname{ball}(H^{\infty})$ and assume that $\varphi$ satisfies \eqref{eqa} and \eqref{eqb}.  Let $u$ be the function defined by \eqref{defn-u}. Then the following assertions are equivalent:
\begin{enumerate}
\item[(i)] $C_{\varphi}: \HH(b) \rightarrow \HH(b)$ is Hilbert--Schmidt.
\item[(ii)] $ \displaystyle{  \int_{ \mathbb{T}}} \dfrac{|u( \xi)|^2}{1- | \varphi( \xi )|^2} dm (\xi )< \infty .$ 
\end{enumerate}
\end{Corollary}
\begin{proof}
Since $(z^n)_{n \geq 0}$ is an orthonormal basis of $H^2$, the operator $W_{u , \varphi}: H^2 \to H^2$ is Hilbert--Schmidt if and only if 
\[
\sum_{n=0}^{ \infty} \|W_{u , \varphi} z^n\|_2^2 < \infty .
\]
But observe that
\[
\sum_{n = 0}^{ \infty} \|W_{u , \varphi} z^n\|_2^2 = \sum_{n = 0}^{ \infty} \|u \varphi^n\|_2^2 = \sum_{n = 0}^{ \infty} \int_{ \mathbb{T}} |u ( \xi)|^2 | \varphi ( \xi)|^{2 n} dm ( \xi) .
\]
By the monotone convergence theorem, we get
\begin{align*}
\sum_{n=0}^{ \infty} \|W_{u, \varphi} z^n \|_2^2 &= \int_{ \mathbb{T}} |u( \xi)|^2 \sum_{n=0}^{ \infty} | \varphi ( \xi)|^{2n} dm ( \xi)\\
&= \int_{ \mathbb{T}} \frac{|u( \xi)|^2}{1 - | \varphi ( \xi)|^2} dm ( \xi).
\end{align*}
The conclusion follows now directly from Theorem~\ref{thm 0.4.0.27}. 
\end{proof}

\subsection{Some necessary/sufficient conditions for compactness}
We have an analogue of Corollary~\ref{Corollary-Gallardo-Partington} for the compactness.
\begin{Corollary}
Let $ \varphi\in\HH(b)\cap \operatorname{ball}(H^{\infty})$ and assume that $\varphi$ satisfies \eqref{eqa} and \eqref{eqb}. Let $u$ be the function defined by \eqref{defn-u} and assume that $u\in H^2$. Assume also that 
\begin{equation}\label{eq:equation-c-delta1}
\supess_{z\in A_\delta}|u(z)|\to 0, \quad \mbox{as }\delta\to 0,
\end{equation}
where $A_\delta:=\{\zeta\in\mathbb T:|\varphi(z)|\geq 1-\delta\}$. Then the operator $C_\varphi$ is compact on $\HH(b)$. 
\end{Corollary}
\begin{proof}
According to Theorem~\ref{thm 0.4.0.27}, the operator $C_\varphi$ is compact on $\HH(b)$ if and only if the operator $W_{u,\varphi}$ is compact on $H^2$. Now the conclusion follows applying \cite[Theorem 2.9]{MR2672342}.
\end{proof}

When we have a compact operator, we can slightly improve the conclusion of Corollary~\ref{cor 0.4.0.21}. 
\begin{Corollary}\label{cor 3.32}
Let $ \varphi\in\HH(b)\cap \operatorname{ball}(H^{\infty})$ satisfying \eqref{eqa} and \eqref{eqb}, and for $1\leq k\leq p$, let $1\leq \ell_k\leq n$ such that $\varphi(\xi_k)=\xi_{\ell_k}$. Assume that the operator $C_\varphi$ is compact on $\HH(b)$. Then for every $1\leq k\leq p$, we have $m_k<m_{\ell_k}$. In particular, the $\xi_k$'s cannot be fixed points for $\varphi$. 
\end{Corollary}
\begin{proof}
Argue by absurd and assume that $C_{ \varphi}$ is compact on $ \HH(b)$ but for some $1\leq k\leq p$, we have $m_{\ell_k}\leq m_k$. For simplicity, we denote $\ell=\ell_k$. On one hand, according to Corollary~\ref{cor 0.4.0.21}, we should have $m_\ell=m_k$ (for this fixed $k$). On the other hand, it follows from Theorem~\ref{thm 0.4.0.27} that the operator $W_{u, \varphi}$ is compact on $H^2$. Hence its adjoint $W_{u, \varphi}^*$ is also compact on $H^2$. Recall now that for every $\lambda\in\D$, we have
\[
W_{u, \varphi}^* k_{ \lambda} = \overline{ u ( \lambda)} k_{ \varphi ( \lambda)},
\]
where $k_\lambda$ is the reproducing kernel of $H^2$ at point $\lambda$, and it is well--known that
\[ 
\frac{k_{z}}{\|k_{z}\|_2} \to 0, \mbox{ weakly in $H^2$,}\qquad \mbox{as }|z|\to 1.
\]
Thus, by compactness, we get
\[
\frac{\|W_{u, \varphi}^* k_{r \xi_k}\|_2}{\|k_{r \xi_k}\|_2} \to 0,\qquad \mbox{as }r\to 1,
\]
from which it follows that
\begin{equation}\label{eq:ddfdsfsdsds422}
|u( r \xi_k)|^2  \frac{1 - r^2}{1 - | \varphi (r \xi_k)|^2} \to 0, \qquad \mbox{as }r\to 1.
\end{equation}
Using the same arguments as in the proof of Corollary \ref{cor 0.4.0.21}, it is easy to check that
\begin{align*}
\left( \frac{  1 - | \varphi(r \xi_k)  |}{1  - r} \right)^{2m_l -1} &\leq 2 C_r  |u( r \xi_k)|^2 \frac{ 1 - r^2}{1  - | \varphi (r \xi_k)|^2} (1-r)^{2m_k-2m_\ell}\\
&=2 C_r  |u( r \xi_k)|^2 \frac{ 1 - r^2}{1  - | \varphi (r \xi_k)|^2},
\end{align*}
where 
\[
C_r =  \frac{ \prod_{i=1, i \neq k}^n |r  \xi_k - \xi_i |^{2 m_i} }{ \prod_{i = 1, i \neq l}^n| \varphi(r  \xi_k) - \xi_i |^{2 m_i}} \frac{1}{ \prod_{j = p+1}^n | \varphi (r \xi_k) - \lambda_j|^{2 m_j}}.
\]
We deduce from \eqref{eq:ddfdsfsdsds422} that 
\[
 \lim_{ r \to 1}\frac{  1 - | \varphi(r \xi_k)  |}{1  - r} = 0 .
\]
But this is in contradiction with  Carath\'eodory's theorem. 
\end{proof}

\begin{Corollary}\label{cor 3.33}
Let $b$ be a rational function in $\operatorname{ball}(H^{\infty})$ (but not a finite Blaschke product) and assume that the zeros $\xi_j$, $1\leq j\leq n$, on $\T$ of its pythagorean mate $a$ all have the same multiplicities. If $C_\varphi$ is compact on $\HH(b)$, then for all $1\leq j\leq n$, we have $\varphi(\xi_j)\in\D$. 
\end{Corollary}
\begin{proof}
Argue by absurd and assume on one hand that the operator $C_{ \varphi}$ is compact on $ \HH(b)$ and on the other hand that there exists some $1 \leq k \leq p$, such that $|\varphi (\xi_k )| = 1$. Then, by Theorem~\ref{lemme 4.0.4.0}, there is some $1\leq \ell_k\leq n$ such that $ \varphi (\xi_k)=\xi_{\ell_k}$. But, by our assumption, we have $m_k = m_{\ell_k}$, which gives a contradiction with Corollary~\ref{cor 3.32}.
\end{proof}

\begin{Remark}
Let $b(z)=(1+z)/2$. According to Corollary~\ref{cor 3.33}, if $\varphi\in\HH(b)\cap \operatorname{ball}(H^{\infty})$ and $\varphi(1)=1$, then the operator $C_\varphi$ is not compact on $\HH(b)$. In particular, we recover a result of D. Sarason and J.N. Silva \cite[Theorem 3.5]{MR1945292} obtained in the context of local Dirichlet space $\mathcal D(\delta_1)$.
\end{Remark}

Corollary~\ref{cor 3.33} might suggest that if a symbol $\varphi$ touches the boundary at some boundary zeros of $a$, then $C_\varphi$ cannot be compact. This is not true as the following example shows.
\begin{Example}
Let $a(z)=c(z-1)(z+1)^2$, $z\in\D$, where $c$ is some suitable constant such that $a\in\operatorname{ball}(H^{\infty})$ and let $b$ be its pythagorean mate. Let us now consider the symbol $\varphi(z)=-(1+z)/2$, $z\in\D$. Of course we have $\varphi(1)=-1$ and $\varphi(-1)=0$. Since $b$ is non-extreme, $\varphi\in\HH(b)\cap \operatorname{ball}(H^\infty)$. Now straightforward computations show that for every $\xi\in\T$, we have 
\[
\frac{|u(\xi)|^2}{1-|\varphi(\xi)|^2}\asymp 1,
\]
from which it follows by Corollary~\ref{thm 4.0.21.0} that the operator $C_\varphi$ is Hilbert--Schmidt on $\HH(b)$. 
\end{Example}

Corollary~\ref{thm 4.0.21.0} enables us to obtain the following  improvement of Lemma~\ref{lem-fonction-analytique-closed-unit-disk}.
\begin{Corollary}\label{cor-fonction-analytique-closed-unit-disk}
Let $b\in \operatorname{ball}(H^{\infty})$ be a rational function (but not a finite Blaschke product) and let $\varphi\in \HH(b)\cap \operatorname{ball}(H^{\infty})$ such that $\|\varphi\|_\infty<1$. Then $C_\varphi$ is Hilbert-Schmidt on $\HH(b)$. 
\end{Corollary}
\begin{proof}
By our assumption,  for every $1\leq j\leq n$, we have $\varphi(\xi_j)\in\D$ and the corresponding function $u$ is defined by
\[
u(z)=\prod_{j=1}^n (\varphi(z)-\xi_j)^{m_j}\prod_{j=1}^n \left(\frac{\varphi(z)-\varphi(\xi_j)}{z-\xi_j}\right)^{m_j}.
\]
Hence, according to Lemma~\ref{lemme 3.9}, this function $u$ belongs to $H^2$. It remains to observe that
\[
\int_\T \frac{|u(z)|^2}{1-|\varphi(\zeta)|^2}\,dm(\zeta)\leq \frac{\|u\|_2^2}{1-\|\varphi\|_\infty^2}<\infty,
\]
and Corollary~\ref{thm 4.0.21.0} implies that $C_\varphi$ is Hilbert--Schmidt on $\HH(b)$. 
\end{proof}

\begin{Remark}
It is not difficult to see that the proof of Theorem~\ref{thm 0.4.0.27} also shows that $C_\varphi$ is of trace--class on $\HH(b)$ if and only if $W_{u,\varphi}$ is of trace--class on $H^2$. Using now \cite[Theorem 2.8]{MR2672342}, we can deduce that when $\|\varphi\|_\infty<1$, the operator $C_\varphi$ is indeed trace--class on $\HH(b)$. 
\end{Remark}

A reinforcement of the condition of Corollary~\ref{cor 0.4.0.22} forces $C_\varphi$ to be compact on $\HH(b)$. 
\begin{Theorem}\label{thm 3.39} 
Let $\varphi\in \HH(b)\cap\operatorname{ball}(H^{\infty})$. Assume that for every $1\leq j\leq n$, we have
 \begin{equation}\label{eq 3.22}
 \underset{z \to \xi_j}{\limsup} | \varphi (z)| < 1,
 \end{equation}
and $ \overline{\varphi ( \mathbb{D})} \cap \mathbb{T} \subset \{\xi_j:1\leq j\leq n\}$. Then $C_{ \varphi}$ is compact on $ \HH(b)$.
\end{Theorem}
\begin{proof}
The assumption~\eqref{eq 3.22} implies that, for every $1\leq j\leq n$, we have $\lambda_j=\varphi(\xi_j)\in \mathbb{D}$. Hence, according to Corollary~\ref{cor 0.4.0.27}, the operator $C_\varphi$ is compact on $\HH(b)$ if and only if 
 \begin{equation}\label{eq111:compact-cas-disque}
 \lim_{|w| \rightarrow 1} \left(\int_{\mathbb{T}} \frac{(1 - |w|^2) |u ( \xi)|^2}{|1 - \overline{w} \varphi ( \xi)|}\,dm ( \xi)  \right) = 0,
 \end{equation}
where $u$ is defined by  $u(z) = \frac{a_1 ( \varphi (z))}{a_1 ( z)} \prod_{i= 1}^n ( \varphi (z) - \varphi ( \xi_i))^{m_i} .$

In order to prove \eqref{eq111:compact-cas-disque}, argue as in the proof of Corollary~\ref{cor 0.4.0.22} to see that there exists $0<L<1$ and $\delta>0$ such that if $V_{\xi_j}=\{\xi\in\T:|\xi-\xi_j|<\delta\}$, $1\leq j\leq n$ and $V=\bigcup_{j=1}^n V_{\xi_j}$, then for almost all $\xi\in V$, we have $|\varphi(\xi)|\leq L$. Thus it follows that, for every  $w \in \mathbb{D}$ and for almost all $\xi \in V$, we have  
\[
|1 - \overline{w} \varphi ( \xi)| \geq 1 - |w| | \varphi ( \xi)| \geq 1 - L > 0.
\]
Hence, for almost all $\xi\in V$, we obtain that
\[
\frac{|u( \xi)|^2}{|1 - \overline{w} \varphi ( \xi)|^2} \leq \frac{\|a_1 \circ \varphi\|_{ \infty}}{(1 -L)^2} |\psi( \xi )|^2 ,
\]
where 
\[
\psi(\xi) = \prod_{j=1}^n \left( \frac{ \varphi (\xi) - \varphi ( \xi_j)}{\xi - \xi_j} \right)^{m_j} , \xi \in \mathbb{T} .
\]
According to Lemma~\ref{lemme 3.9}, since $ \varphi \in \HH(b)$, we know that $\psi \in H^2$ and we deduce that
\[
\int_{V} \frac{(1 - |w|^2) |u( \xi)|^2 }{|1 - \overline{w} \varphi(\xi)|^2} dm ( \xi) \leq (1 - |w|^2) \frac{\|a_1 \circ \varphi\|_{ \infty}}{(1 -L)^2}\|\psi\|_2^2,
\]
from which it immediately follows that
\[
\lim_{|w| \to 1 } \left(\int_V \frac{(1 - |w|^2) |u( \xi)|^2}{|1 - \overline{w} \varphi ( \xi)|^2} dm ( \xi)  \right) = 0.
\]
In order to prove \eqref{eq111:compact-cas-disque}, it remains to check that
\begin{equation}\label{eq222:compact-cas-disque}
\lim_{|w| \to 1 } \left(\int_{ \mathbb{T} \setminus V} \frac{(1 - |w|^2) |u( \xi)|^2}{|1 - \overline{w} \varphi ( \xi)|^2} dm ( \xi)  \right) = 0. 
\end{equation}
To this purpose, observe that for $ \xi \in \mathbb{T} \setminus V$, we have for every $1\leq j\leq n$, $|\xi - \xi_j| \geq \delta$, whence 
\[
 |a_1 ( \xi)| = \prod_{j=1}^n |\xi - \xi_j|^{m_j} \geq { \delta}^N.
 \]
Moreover, for almost all $\xi \in \mathbb{T} \setminus V$, we have
\[
\prod_{j=1}^n | \varphi ( \xi) - \varphi ( \xi_j)|^{m_j} \leq 2^N,
\]
which gives that 
\[
|u( \xi)| \leq \left( \frac{2}{\delta} \right)^N |(a_1 \circ \varphi)( \xi)|.
\]
Hence, we obtain
\begin{align*}
\int_{\mathbb{T} \setminus V} (1 - |w|^2) \frac{| u (\xi)|^2}{|1 - \overline{w} \varphi ( \xi)|^2} dm( \xi) & \leq \left(\frac{2}{\delta} \right)^{2 N} \int_{ \mathbb{T} \setminus V }  \frac{ (1 - |w|^2)| a_1 ( \varphi ( \xi))|^2}{|1 - \overline{w} \varphi ( \xi)|^2}\,dm( \xi)\\
& \leq \left(\frac{2}{\delta} \right)^{2 N} \int_{ \mathbb{T} }  \frac{ (1 - |w|^2) \prod_{j=1}^n  | \varphi ( \xi) - \xi_j|^{2m_j}}{|1 - \overline{w} \varphi ( \xi)|^2}\,dm( \xi).
\end{align*}
In particular, in order to prove \eqref{eq222:compact-cas-disque}, it is sufficient to check that
\begin{equation}\label{eq 3.23}
\lim\limits_{|w| \to 1} \left( \int_{ \mathbb{T} }  \frac{ (1 - |w|^2) \prod_{j=1}^n  | \varphi ( \xi) - \xi_j|^{2m_j}}{|1 - \overline{w} \varphi ( \xi)|^2} dm( \xi) \right) = 0 . 
\end{equation}
Since $ \varphi \in \HH(b) \subset H^2$, the function $\varphi$ has radial limits at almost all points in $\mathbb{T}$.  Denote by 
\[
E= \left\{ \xi \in \mathbb{T} : \varphi ( \xi) = \lim\limits_{ r \to 1^-} \varphi( r \xi) \hspace{0.1 cm} \text{exists}  \right\}.
\]
We have $m( \mathbb{T} \setminus E) = 0 $ and then 
\[
\int_{ \mathbb{T}} \frac{ (1 - |w|^2) \prod_{j=1}^n  | \varphi ( \xi) - \xi_j|^{2m_j}}{|1 - \overline{w} \varphi ( \xi)|^2}\,dm( \xi) = \int_E \frac{ (1 - |w|^2) \prod_{j=1}^n  | \varphi ( \xi) - \xi_j|^{2m_j}}{|1 - \overline{w} \varphi ( \xi)|^2}\,dm( \xi).
\]
Fix now $0<\varepsilon<1$ and for $1 \leq j \leq n$, introduce the following subsets of $E$ defined by
 \[
W_j = \{ \xi \in E : | \varphi ( \xi) - \xi_j| \leq \varepsilon \} \quad\mbox{and}\quad W= \bigcup_{j=1}^n W_j.
\]
Note that for $ \xi \in W$, we have
$$ \prod_{j=1}^n | \varphi( \xi) - \xi_j|^{2 m_j} \leq \varepsilon^2 4^N ,$$
whence
 \begin{align*}
\int_W \frac{ (1 - |w|^2) \prod_{j=1}^n  | \varphi ( \xi) - \xi_j|^{2m_j}}{|1 - \overline{w} \varphi ( \xi)|^2} dm( \xi)  & \leq 4^N \varepsilon^2 \int_{ \mathbb{T}} \frac{(1 -|w|^2 )}{ |1 - \overline{w} \varphi ( \xi)|^2} dm( \xi)\\
&= 4^N \varepsilon^2 ( 1 - |w|^2) \|C_{ \varphi} k_w\|_2^2 \\
& \leq 4^N \varepsilon^2 \|C_{ \varphi}\|_{ \mathcal{L} (H^2)}^2.
\end{align*} 
For the integral on $E \setminus W$, observe that 
\[
L_{ \varepsilon} : = \sup_{  \xi \in E \setminus W} |\varphi ( \xi)| < 1.
\]
Indeed argue by absurd and assume that $L_{ \varepsilon} = 1$. Then there exists $(\zeta_k)_k \subset E \setminus W$ such that $  |\varphi ( \zeta_k)|\to 1$, as $k \to \infty$. By compactness, we may assume that $\varphi( \zeta_k)\to e^{i \theta}$, as $k\to\infty$, for some $\theta \in \mathbb{R}$. Hence, for every $\eta>0$, there exists $N\in\mathbb N$ such that 
 \[
k \geq N \implies | \varphi ( \zeta_k) - e^{i \theta}| < \frac{\eta}{2}.
\]
Since $ \zeta_k \in E$, there is  $0 < r_k < 1$ such that 
$$ | \varphi(r_k \zeta_k) - \varphi( \zeta_k)| < \frac{\eta}{2},$$
which gives that 
\[
k \geq N \implies |\varphi(r_k \zeta_k) - e^{i \theta}| \leq \eta.
\]
In particular, $e^{i \theta} \in \overline{\varphi ( \mathbb{D})} \cap \mathbb{T}$ and then our assumption implies that there exists $1 \leq j \leq n$ such that $e^{ i \theta} = \xi_j$.
Therefore, $\varphi ( \zeta_k ) \to \xi_j$, as $k\to\infty$, which contradicts the fact that $ \zeta_k \in E \setminus W$. Thus, $L_{ \varepsilon} < 1$ and for every $\xi \in E \setminus W$, we have 
\[
\frac{ (1 - |w|^2) \prod_{j=1}^n  | \varphi ( \xi) - \xi_j|^{2m_j}}{|1 - \overline{w} \varphi ( \xi)|^2}  \leq \frac{4^N}{(1- L_{ \varepsilon})^2} ( 1 - |w|^2 ).
\]
We then deduce that
\[
\int_{  E \setminus W}  \frac{ (1 - |w|^2) \prod_{j=1}^n  | \varphi ( \xi) - \xi_j|^{2m_j}}{|1 - \overline{w} \varphi ( \xi)|^2}\,dm( \xi)   \leq \frac{4^N}{(1- L_{ \varepsilon})^2} ( 1 - |w|^2 ).
\]
Now, one can choose $0 < r_0 < 1$ such that
\[ 
r_0 < |w| < 1 \implies  \frac{4^N}{(1- L_{ \varepsilon})^2} ( 1 - |w|^2 ) \leq \varepsilon^2,
\]
which gives  
\[
\int_{ \mathbb{T}}  \frac{ (1 - |w|^2) \prod_{j=1}^n  | \varphi ( \xi) - \xi_j|^{2m_j}}{|1 - \overline{w} \varphi ( \xi)|^2}\,dm( \xi)   \leq \left( 4^N \|C_{ \varphi}\|_{ \mathcal{L}(H^2)}^2 +1 \right) \varepsilon^2.
\]
Finally \eqref{eq 3.23} is proved, which concludes the proof. 
\end{proof}

\begin{Example}
Let $ b(z)=(1+z^2)/2$, $z \in \mathbb{D}$. Then, we easily check that its pythagorean mate is $a(z)=(1-z^2)/2$, $z \in \mathbb{D}$. For $\gamma > 1/2$, consider $\varphi(z)=\left(\frac{1 - z^2}{2} \right)^{ \gamma}$, $z \in \mathbb{D}$.
Observe that
\[ 
\frac{\varphi(z)}{a(z)} = \frac{1}{  \left( \frac{1-z^2}{2} \right)^{1 - \gamma}} , \qquad z \in \mathbb{D},
\]
and  since $\gamma>1/2$, we have $ \varphi \in a H^2 \subset \HH(b)$. Moreover, $ \underset{z \to \pm 1 }{\lim} \varphi (z) = 0 $. Finally, we see that 
\[
| \varphi(z)| = 1 \Leftrightarrow  \left|\frac{1- z^2}{2} \right| = 1 \Leftrightarrow z^2 = - 1 \Leftrightarrow z = \pm i.
\]
But $ \varphi ( \pm i) = 1$ and then $ \overline{ \varphi ( \mathbb{D})} \cap \mathbb{T} = \{1\}$. Thus, we can apply Theorem~\ref{thm 3.39} which implies that the operator $C_{ \varphi}$ is compact on $ \HH(b)$.

\end{Example}

\subsection{Hilbert--Schmidt property in some particular cases}
\begin{Corollary}\label{cor 3.35}
Let $b(z)=(1+z)/2$ and let $\varphi\in \HH(b)\cap \operatorname{ball}(H^{\infty})$. The following assertions are equivalent:
\begin{enumerate}
\item[(i)] $C_{ \varphi}$ is Hilbert-Schimdt on $ \HH(b)$.
\item[(ii)] $ \varphi(1) \in \mathbb{D}$ and
$$\int_{ \mathbb{T}} \frac{| \varphi ( \xi) - 1|^2 | \varphi (\xi) - \varphi (1)|^2}{| \xi - 1|^2 ( 1 - | \varphi ( \xi)|^2)}dm( \xi) < \infty .$$
\end{enumerate}
\end{Corollary}
\begin{proof}
On one hand, recall that when $b(z)=(1+z)/2$, then $a_1(z)=z-1$. On the other hand, observe that, according to Corollary~\ref{cor 3.33}, the condition $ \varphi(1) \in \mathbb{D}$ is necessary for the compactness of $C_{ \varphi}$ on $\HH(b)$. Now, it follows from Corollary~\ref{thm 4.0.21.0} that the operator $C_{ \varphi}$ is Hilbert--Schmidt  on $\HH(b)$ if and only if 
\[
\int_{ \mathbb{T}} \frac{|u( \xi)|^2}{1 - | \varphi ( \xi)|^2} dm ( \xi) < \infty,
\]
where  
\[
u(z) = \frac{a_1 ( \varphi (z))}{a_1 (z)} ( \varphi (z) - \varphi (1))=\frac{(\varphi(z)-1)(\varphi(z)-\varphi(1)}{z-1},
\]
which gives the result.
\end{proof}

\begin{Example}
Let $b(z)=(1+z)/2$, $z\in\D$. Then $a(z)=(1-z)/2$, $z\in\D$. It is known that the operator $C_a$ is not compact on $H^2$ (see \cite{MR1237406}), whereas, according to Corollary~\ref{cor 3.35}, the operator $C_a$ is Hilbert-Schmidt on $\HH(b)$. Indeed, on one hand, we have $a(1)=0$. On the other hand, easy computations show that 
\[
\frac{|a( \xi) - 1|^2 |a(\xi) - a(1)|^2}{| \xi - 1|^2 ( 1 - |a( \xi)|^2)}=\frac{1}{4},
\]
whence assertion (ii) of Corollary \ref{cor 3.35} is trivially satisfied. Thus the operator $C_a$ is Hilbert--Schmidt on $\HH(b)$. 
\end{Example}

\begin{Example}
Let $b(z)=(1-z^2)/2$, $z\in\D$. Then, it is easy to see that $a(z)=(1+z^2)/2$, $z\in\D$, and thus $a_1(z)=z^2+1=(z-i)(z+i)$, $z\in\D$. According to Corollary~\ref{cor 0.4.0.22}, observe that $C_a$ is bounded on $\HH(b)$, because 
\[
\limsup_{z\to\pm i}|a(z)|=|a(\pm i)|=0.
\]
Now straightforward computations show that 
\[
\frac{|u(\xi)|^2}{1-|a(\xi)|^2}=\frac{|u(\xi)|^2}{|b(\xi)|^2}\asymp \frac{|1+a(\xi)^2|^2|1+\xi^2|^2}{|1-\xi^2|^2}.
\]
In a neighborhood of $1$, we then have
\[
\frac{|u(\xi)|^2}{1-|a(\xi)|^2}\asymp \frac{1}{|1-\zeta|^2},
\]
from which it follows that 
\[
\int_{\T}\frac{|u(\xi)|^2}{1-|a(\xi)|^2}\,dm(\xi)=\infty.
\]
Hence, according to Corollary~\ref{thm 4.0.21.0} the operator $C_a$ is not Hilbert--Schmidt on $\HH(b)$. 
\end{Example}

The previous examples motivate the following question.
\begin{Question}\label{question2}
Let $b\in \operatorname{ball}(H^{\infty})$ be a rational function (but not a finite Blaschke product), and let $a$ be its pythagorean mate. Is $C_a$ compact on $\HH(b)$? 
\end{Question}
Observe that according to Corollary~\ref{cor 0.4.0.22}, $C_a$ is at least always bounded on $\HH(b)$ because, for all $1\leq j\leq n$, we have
\[
\limsup_{z\to\xi_j}|a(z)|=|a(\xi_j)|=0.
\]

\bibliographystyle{plain}

\bibliography{bibliographie}

\end{document}